\documentclass[a4paper, 11pt]{amsart}
\pdfoutput=1 
\usepackage{custom-template, shortcuts}

\begin{document}
	
\title{Ranks of abelian varieties in cyclotomic twist families}
\author{Ari Shnidman and Ariel Weiss}
\date{}

\address{Ari Shnidman, \newline Einstein Institute of Mathematics, The Hebrew University of Jerusalem, Edmund J.\ Safra Campus, Jerusalem 9190401, Israel.\vspace*{-6pt}}
\email{ariel.shnidman@mail.huji.ac.il\vspace*{-6pt}}
\address{Ariel Weiss,\newline Department of Mathematics, The Ohio State University, Columbus, OH, USA\newline Department of Mathematics, Ben-Gurion University of the Negev, Be'er Sheva 8410501, Israel.\vspace*{-6pt}}
\email{weiss.742@osu.edu}

\begin{abstract}
Let $A$ be an abelian variety over a number field $F$, and suppose that $\Z[\zeta_n]$ embeds in $\End_{\bar F} A$, for some root of unity $\zeta_n$ of order $n = 3^m$.  Assuming that the Galois action on the finite group $A[1-\zeta_n]$ is sufficiently reducible, we bound the average rank of the Mordell--Weil groups $A_d(F)$, as $A_d$ varies through the family of $\mu_{2n}$-twists of $A$.  Combining this result with the recently proved uniform Mordell--Lang conjecture, we prove near-uniform bounds for the number of rational points in twist families of bicyclic trigonal curves $y^3 = f(x^2)$, as well as in twist families of theta divisors of cyclic trigonal curves $y^3 = f(x)$. Our main technical result is the determination of the average size of a $3$-isogeny Selmer group in a family of $\mu_{2n}$-twists.
\end{abstract}

\maketitle
\vspace{-20pt}
\section{Introduction}
Let $A$ be an abelian variety over a number field $F$ and let $G_F = \Ga F$. Any $G_F$-stable subgroup $H \subset \Aut_{\bar F} A$ gives rise to a twist family of abelian varieties $A_\xi$ over $F$, indexed by the elements $\xi$ of the Galois cohomology set $H^1(G_F, H)$. The base change of $A_\xi$ to $\overline F$ is isomorphic to $A_{\bar F}$, but with $G_F$-action twisted by $\xi$. Our goal is to study the distributions of the ranks of the Mordell--Weil groups $A_\xi(F)$ in such twist families, and to  give some applications.

Every abelian variety $A$ has the automorphism $-1$, and since $H^1(G_F, \{\pm 1\}) \simeq F^\times/F^{\times2}$, we obtain the \emph{quadratic twist family} of $A$. The average rank of $A_\xi(F)$ in quadratic twist families has been extensively studied in the case of elliptic curves \cites{Brumer,Heath-Brown,Katz-Sarnak,smith,bkls}, with \cite{bkls} addressing many cases in higher dimension as well. 

In this paper, we consider the case $H = \mu_{2n}$, the group of $2n$-th roots of unity, where $n = 3^m$ for some $m \geq 1$. More precisely, we assume that there is a $G_F$-equivariant ring embedding $\Z[\zeta] \hookrightarrow \End_{\bar F}A$, where $\zeta = \zeta_n \in \overline F$ is a root of unity of order $n$. We say that such an $A$ has \emph{$\zeta$-multiplication}. For example,  the Jacobian $J$ of a curve of the form $y^3 =xf(x^{3^{m-1}})$ has $\zeta$-multiplication induced from the order $n$ automorphism $(x, y)\mapsto (\zeta^3x, \zeta y)$. 

Since $\mu_{2n} = \langle -\zeta\rangle \subset \Aut_{\bar F}A$, there is a twist $A_d$ for each $d \in  F^\times/F^{\times 2n} \simeq H^1(G_F, \mu_{2n})$. In the example above, $J_d$ is the $d$-th quadratic twist of the Jacobian of $y^3 = xf(\frac 1d x^{3^{m-1}})$. In \Cref{sec:avg-size-selmer}, we recall a height function $h \colon F^\times/F^{\times 2n} \to \R$ with the property that the sets $\Sigma_X := \{d \in F^\times/F^{\times2n} : h(d) < X\}$ are finite. When $F = \Q$, the height $h(d)$ is the absolute value of the smallest integer representing $d$. The average rank of $A_d(F)$ is then, by definition, 
\[\avg_d \rk A_d(F) = \lim_{X \to \infty} \underset{d \in \Sigma_X}\avg \rk A_d(F). \]
In general, it is not known whether this limit exists or even if the limsup is finite. In the latter case, we say that the \emph{average rank of $A_d(F)$ is bounded}.

\subsection{Mordell--Weil ranks}
If $A$ has $\zeta$-multiplication, the endomorphism $1 - \zeta \in \End_{\bar F}A$ descends to an isogeny $\pi \colon A \to A'$ over $F$ (see \Cref{sec:zeta-mult}). The kernel $A[\pi]$ is a $G_F$-stable subgroup of the 3-torsion group $A[3]$, and hence is a finite dimensional $\F_3$-vector space. We show that the average rank of $A_d(F)$ is bounded, assuming that the $G_F$-action on $A[\pi]$ is sufficiently reducible.

\begin{theorem}\label{thm:completely-reducible}
Let $A$ be an abelian variety with $\zeta_{3^m}$-multiplication over $F$.
\begin{enumerate}
    \item If $A[\pi]$ is a direct sum of characters, then $\avg_d \rk A_d(F)$ is bounded.
    \item If $A[\pi]$ has a full flag, then $\avg_d \rk A_d(F)$, over squarefree $d \in F^\times/F^{\times 2n}$, is bounded. 
\end{enumerate}
\end{theorem} 

Here, we say that $A[\pi]$ \emph{has a full flag} if there are $G_F$-modules $0 = H_0 \subset H_1 \subset \cdots \subset H_k = A[\pi]$ such that $\dim_{\F_3} H_{i+1}/H_i = 1$.  We say that $d\in F^\times/F^{\times 2n}$ is \emph{squarefree} if $v(d) \equiv 0$ or $1\pmod {2n}$, for all finite places $v$ of $F$. \Cref{thm:completely-reducible} is a simultaneous generalization of \cite[Thm.\ 2.2]{bkls} and \cite[Thm.\ 5]{elk} to a larger class of twist families of abelian varieties.

If $J$ is the Jacobian of $y^3 = xf(x^{3^{m-1}})$, then the representation theoretic conditions on $J[\pi]$ translate into  conditions on the Galois group $\Gal(f)$ of the splitting field of $f(x)$ over $F$. More generally, we deduce the following result from  \Cref{thm:completely-reducible}:

\begin{corollary}\label{cor:trigonal}
Let $f(x) \in F[x]$ be separable and non-constant, and let $J$ be the Jacobian of either $y^{3^m} = f(x)$ or $y^3 = xf(x^{3^{m-1}})$. 
\begin{enumerate}
    \item If $\Gal(f) \simeq (\Z/2\Z)^k$, for some $k \geq 0$, then $\avg_d \rk J_d(F)$ is bounded.  
    \item If $\Gal(f)$ is an extension of $(\Z/2\Z)^k$ by a $3$-group, then $\avg_d\rk J_d(F)$, over squarefree $d\in F^\times/F^{\times 2n}$, is bounded.
\end{enumerate}
\end{corollary}

Our proof of \Cref{thm:completely-reducible} gives an explicit upper bound on the average rank, however, the bound depends on subtle arithmetic properties of $A$. The following crude upper bound has the virtue that it applies to a large class of abelian varieties and depends only mildly on $A$. 

\begin{theorem}\label{thm:explicit-rank-bound-intro}
Suppose that $A[\pi]$ has a full flag and that $A$ admits a $\zeta_n$-stable principal polarization. Let $S$ be the set of places of $F$ dividing  $3\f_A\infty$, where $\f_A$ is the conductor of $A$. Then the average rank of $A_d(F)$, for squarefree $d\in F^\times/F^{\times 2n}$, is at most $\dim A\cdot (\#S + 3^{-\#S})$.
\end{theorem}

For most $A$, this bound is significantly  weaker than what our method actually gives. An interesting case is when $A$ has complex multiplication (CM), i.e.\ $\dim A = 3^{m-1}$, in which case $\dim_{\F_3} A[\pi] = 1$ and the reducibility hypotheses are automatically satisfied. When the complex multiplication is defined over $F$, we obtain especially strong results:

\begin{theorem}\label{thm:cmintro}
Suppose that $\dim A = 3^{m-1}$, so that $A$ has complex multiplication by $\Z[\zeta_{3^m}]$.  Assume moreover that $\zeta_{3^m} \in F$, so that the complex multiplication is defined over $F$.  Then the average $\Z[\zeta_{3^m}]$-module rank of $A_d(F)$ is at most $\frac12$, and at least $50\%$ of twists $A_d$ have rank $0$.
\end{theorem}

In the CM case, we expect that $100\%$ of twists $A_d$ have rank $0$, in which case our result is halfway towards the analogue of Goldfeld's conjecture in this context.

Such $A$ arise as factors of the Jacobians of the curves $y^{3^m} = x^a(1-x)^b$, which have good reduction away from $3$.  Even when the CM is not defined over $F$, we obtain average rank bounds that depend only on $\dim A$. Over $\Q$, for example, the average rank of $A_d(\Q)$ is at most $\frac{19}{9} \dim A$ by \Cref{thm:explicit-rank-bound-intro}, a bound which can be improved to $\frac{13}{9} \dim A$ with a more refined analysis.    

\subsection{Rational points on curves}\label{subsec:qmintro}

 \Cref{thm:completely-reducible} has concrete consequences for the  arithmetic of curves $C/F$ of genus $g \geq 2$.  It was nearly $40$ years ago that Faltings proved that $C(F)$ is a finite set, but very recently, there has been significant progress towards a uniform upper bound on $\#C(F)$.  Building on work of Dimitrov, Gao, and Habegger \cite{DGH}, K\"uhne has shown that 
\[\#C(F) \leq c_g^{1 + \rk \Jac(F)},\]
 where $c_g$ is a constant depending only $g$ \cite{kuhne}.  Building on this work, Gao, Ge, and K\"uhne proved the more general uniform Mordell--Lang conjecture \cite{GGK} for closed subvarieties of abelian varieties.  These results reduce the question of uniform bounds for rational points on a large class of varieties to a question about ranks of abelian varieties. 
 
 By combining these results with \Cref{thm:completely-reducible}, we show that ``near-uniformity'' holds for twists of bicyclic trigonal curves:
 
 \begin{theorem}\label{thm:uniformcurves}
 Let $f(x) \in F[x]$ be separable, of degree at least two, and with all of its roots non-zero elements of $F$. Consider the  bicyclic trigonal curve $C\colon y^3 = f(x^2)$,  and let $C_d \colon dy^3 = f(dx^2)$ be the corresponding sextic twist family. Then for every $\epsilon > 0$, there is a constant $N_\epsilon$ such that the lower density of classes $d \in F^\times/F^{\times6}$ for which \[\#C_d(F) \leq N_\epsilon\]
 is at least $1-\epsilon$.
 \end{theorem}

To prove \Cref{thm:uniformcurves}, we apply \Cref{thm:completely-reducible} not to the Jacobian of $C$, but to the Prym variety for the double cover $C \to C'$, where $C' \colon y^3 = f(x)$; see \Cref{sec:uniform}. We remark that the constant $N_\epsilon$ depends only on $\epsilon$, $\deg(f)$, and $\#S$ (using the notation of \Cref{thm:explicit-rank-bound-intro}).

Unlike the curves in \Cref{thm:uniformcurves}, a general cyclic trigonal curve $C \colon y^3 = f(x)$ has no sextic twists, so it may seem that \Cref{thm:completely-reducible} says nothing about rational points on the twists of $C$ itself. However, for these curves, we can consider sextic twists of a theta divisor $\Theta \subset J = \Jac(C)$. Recall that $\Theta$ is birational to the symmetric power $C^{(g-1)}$, so its rational points parameterize low degree points on $C$. We can choose $\Theta$ so that it is preserved by the $\mu_{2n}$-action (see \Cref{sec:uniform}), which allows us to consider the twist $\Theta_d \subset J_d$, for each $d \in F^\times$.      

\begin{theorem}\label{thm:uniformtheta}
 Let $f(x) \in F[x]$ be separable and suppose that $\Gal(f)\simeq (\Z/2\Z)^{k}$ for some $k \geq 0$. Let $C\colon y^3 = f(x)$, and suppose that $\Jac(C)$ is geometrically simple. Let $\Theta \subset \Jac(C)$ be a symmetric theta divisor.  Then for every $\epsilon > 0$, there is a constant $N_\epsilon$ such that the lower density of classes $d \in F^\times/F^{\times6}$ for which
 \[\#\Theta_d(F) \leq N_\epsilon\]
 is at least $1-\epsilon$.
\end{theorem} 
 
This result again follows from \Cref{thm:completely-reducible} and \cite{GGK}, and  $N_\epsilon$ depends only on $\epsilon$, $\deg(f)$, and $\#S$. Since the results of \cite{GGK} are ineffective, we cannot say anything explicit about the constant $N_\epsilon$ in general. However, one can prove explicit results in this direction by instead combining our work with the Chabauty method. We illustrate this by way of an example. 

\begin{theorem}\label{thm:genus3example}
Consider the sextic twist family $C_d \colon y^3 = (x^2 - d)(x^2 - 4d)$ of genus $3$ curves. For at least $\frac13$ of squarefree $d \in \Z$ such that $d \equiv 2$ or $11 \pmod{36}$, we have $\#C_d(\Q) \leq 5$. 
\end{theorem}

The curve $C_d$ admits a double cover $p \colon C_d \to E_d$ to the elliptic curve $E_d \colon y^3 = (x-d)(x - 4d)$. Moreover, $C_d$ embeds in the abelian surface $P_d = \Jac(C_d)/p^*\Pic^0(E_d)$. By making the rank bound in \Cref{thm:completely-reducible} explicit, we show that $\rk P_d(\Q) \leq 1$ for at least $\frac13$ of twists $d$. Then we invoke, and generalize slightly, Stoll's uniform Chabauty result for twist families \cite{stoll:independence}. 

The same method works for sextic twist families of the form $C_{a,d} \colon y^3 = (x^2 - d)(x^2 - ad)$. To prove the existence of twists with $\rk P_{a,d}(\Q) \leq 1$, we must check that a certain local $3$-adic root number takes the value $-1$ for some twist $d$. We can verify this condition in Magma for seemingly any given curve $C_{a,d}$, but it would be nice to have a proof for all or most values of $a$.\footnote{In \cite{SW-rank-gain}*{Thm.\ 1.4}, we prove that a positive proportion of $P_{a,d}$ have rank at most $1$ in the case that $a$ is a square, using a different argument which sidesteps the root number question.} It would also be interesting to prove explicit results for symmetric squares of  trigonal plane quartics, as in \Cref{thm:uniformtheta}, by using the recent work of Caro--Pasten \cite{caro-pasten}.  

\subsection{3-isogeny Selmer groups}\label{subsec:3-isog}

Having discussed some applications of \Cref{thm:completely-reducible}, let us discuss its proof.  \Cref{thm:completely-reducible} follows from a more precise result about Selmer groups. Let $A/F$ be an abelian variety with $\zeta$-multiplication and admitting a $3$-isogeny $\phi \colon A \to B$. If $A[\phi] \subset A[\pi]$, or equivalently, if $\phi$ is $\zeta$-linear, then each twist $A_d$ is endowed with its own $3$-isogeny $\phi_d \colon A_d \to B_d$. For each $d$, we consider the $\phi_d$-Selmer group $\Sel_{\phi_d}(A_d)$, which sits in the exact sequence
\[0 \to B_d(F)/\phi_d A_d(F) \to \Sel_{\phi_d}(A_d) \to \Sha(A_d)[\phi_d] \to 0.\] 
The main technical result of this paper is the exact computation of $\avg_d \#\Sel_{\phi_d}(A_d)$. 

To state the precise result, we recall the global Selmer ratio $c(\phi_d) = \prod_v c_v(\phi_d)$, where for each place $v$ of $F$, we define 
\[c_v(\phi_d) = \dfrac{\#\mathrm{coker}\left(A_d(F_v) \to B_d(F_v)\right)}{\#\mathrm{ker}\left(A_d(F_v) \to B_d(F_v)\right)}.\]
 For $v \nmid 3\infty$, we have $c_v(\phi_d) = c_v(B_d)/c_v(A_d)$, where $c_v(A)$ is the Tamagawa number of $A$ over $F_v$. Thus, up to some subtle factors at places $v$ above $3$ and $\infty$, the number $c(\phi_d)$ is the ratio of the global Tamagawa numbers $c(B_d)/c(A_d)$. In particular, we have $c_v(\phi_d) \in 3^\Z$, and $c_v(\phi_d) = 1$ for all but finitely many $v$ (having fixed $d$).

We say $A[\phi]$ is \emph{almost everywhere locally a direct summand} of $A[\pi]$ if, for almost all places $v$ of $F$, the $G_{F_v}$-module $A[\phi]$ is a direct summand of $A[\pi]$.

\begin{theorem}\label{thm:selmer}
 Assume that $A[\phi]$ is almost everywhere locally a direct summand of $A[\pi]$. Then $\avg_d \#\Sel_{\phi_d}(A_d) = 1 + \avg_d c(\phi_d)$, where both averages are finite and taken over $d \in F^\times/F^{\times 2n}$, ordered by height.  
\end{theorem}

This result is a simultaneous generalization of \cite[Thm 2.1]{bkls} and \cite[Thm.\ 1]{elk}.  Interestingly, the condition of being everywhere locally a direct summand, which is automatically satisfied for the families considered in \cites{bkls,elk}, seems to be an obstruction to computing the average size of $\#\Sel_{\phi_d}(A_d)$ in the entire family of twists, at least using our methods. 
In any case, if we only consider squarefree twists, then this obstruction goes away:

\begin{theorem}\label{thm:squarefreeselmer}
Let $\phi \colon A \to B$ be a $\zeta$-linear $3$-isogeny. Then the average size of $\#\Sel_{\phi_d}(A_d)$ over squarefree $d \in F^\times/F^{\times 2n}$ is finite and equal to $1 + \avg_d c(\phi_d)$. 
\end{theorem}

The quantity $\avg_d c(\phi_d)$ is governed by local arithmetic data which can be made explicit in certain cases. For example, in \Cref{thm:cmintro} we have $c(\pi_d) = c(\phi_d) = 1$, for all $d$. 
However, in general, computing the exact value of $\avg_d c(\phi_d)$ is hard. Nonetheless, one can give an explicit upper bound on $\avg_d \#\Sel_{\phi_d}(A_d)$ depending only on $F$, $\dim A$, and the number of primes dividing the conductor of $A$ (\Cref{prop:Selmer-rank-explicit-bound}).

In \Cref{sec:twists-of-ab-vars} we show how to deduce \Cref{thm:completely-reducible} from Theorems \ref{thm:selmer} and \ref{thm:squarefreeselmer}. In the remainder of the introduction, we discuss the proofs of the latter two results.

\subsection{Methods}

We prove Theorems \ref{thm:selmer} and \ref{thm:squarefreeselmer} using geometry-of-numbers methods. As in the previous works \cites{bkls, elk} of the first author and his collaborators, we first identify the elements of $\Sel_{\phi_d}(A_d)$ with $\SL_2(F)$-orbits of binary cubic forms of discriminant $d$. We then wish to use lattice-point counting techniques, which have been extended to global fields in \cite{BSW:globalI}, to count the number of such orbits of bounded discriminant. However, it is not at all clear (and indeed, it is not always true) that the $\SL_2(F)$-orbits corresponding to Selmer elements are \emph{integral}, i.e.\ that they contain cubic forms whose coefficients are algebraic integers. This integrality is of course essential for lattice-point counting.

For the quadratic twist families considered in \cite{bkls}, integrality follows quickly once it is realized that the local Selmer conditions are very mild outside the finitely many primes dividing the conductor of $A$. One needs only to ``clear denominators'' at those finitely many primes, and then the Selmer orbits become integral. For the families considered in this paper, the question of integrality is more subtle. The first interesting case is the family of sextic twists $E \colon y^2 = x^3 + d$ considered in \cite{elk}, which is the unique twist family of elliptic curves with $3$-power $\zeta$-multiplication. In that special family, the authors give an \emph{explicit} bijection between Selmer elements and binary cubic forms. Integrality is then proven using a direct connection with Bhargava's higher composition laws. 

In the more general setup of this paper, we cannot rely on such an explicit parameterization, nor do we expect one to exist. Instead, our method is more abstract and involves two independent steps. First, we give a complete analysis of the \emph{integral} arithmetic invariant theory for the representation $\Sym^3\Z^2$ of $\SL_2$, in the sense that we identify precisely which $\SL_2(F_v)$-orbits of binary cubic forms, over a local field $F_v$, have integral representatives. We show that once the valuation of the discriminant $v(d)$ is at least $3$, all orbits have integral representatives, and we determine what happens for $v(d) \leq 2$ as well. Our strategy is to translate the question into one about cubic rings, whose local structure we understand well. Second, we study the Selmer groups $\Sel_{\phi_d}(A_d)$ from a purely cohomological point of view, in the spirit of Mazur and Rubin \cite{MazurRubinFindingLarge,KMR}.  Upon comparing the results, we find that for all but finitely many primes $v$, the Selmer orbits of discriminant $d$ are $v$-integral, except possibly when $v(d) = 2$. In particular, the Selmer orbits are integral when $d$ is squarefree. When $v(d) = 2$, we find that the local direct summand condition on $A[\phi] \subset A[\pi]$ exactly matches up with the local integrality condition. 

\subsection{Future directions}
In many situations, the integral orbits of a reductive group $G$ acting on a representation $V$ have been shown to parameterize Selmer elements in a certain explicit family of abelian varieties \cite{BS2Selmer,BS3Selmer,BhargavaHo,BhargavaGross,Thorne:Vinberg,laga}. The results of this paper show that there is a tremendous amount of flexibility in these constructions, in the sense that $(G,V)$ can be used to parameterize Selmer elements in very different looking families over the same space of invariants $V/\!\!/G$ (which is $\A^1$ in our case). Our analysis of the integral arithmetic invariant theory of $(\SL_2, \Sym^3 \Z^2)$ can be adapted to some of these other representations $(G,V)$, so it would be interesting to understand the following (vaguely formulated) question. For which families $\phi \colon \mathcal{A} \to \mathcal{B}$ of isogenies of abelian varieties over $S = V /\!\!/ G$ can the elements of $\Sel_{\phi_s}(\mathcal{A}_s)$, for $s \in S(F)$, be parameterized by orbits of $G(F)$ on $V(F)$?  

\subsection{Outline}

We begin in \Cref{sec:zeta-mult} with basics on abelian varieties with $\zeta$-multiplication. Sections \ref{sect: integral}-\ref{sec:twists-of-ab-vars} are the technical heart of the paper. In \Cref{sect: integral}, we give a complete analysis of the integral arithmetic theory for the representation $\Sym^3\Z^2$ of $\SL_2$. In \Cref{sec:localselmer}, we give a parallel, but independent analysis of the Selmer groups $\Sel_{\phi_d}(A_d)$ of the twists $\phi_d$ of a general $\zeta$-linear $3$-isogeny. In \Cref{sec:selmer-groups-and-integrality}, we combine these two sections and prove Theorems \ref{thm:selmer} and \ref{thm:squarefreeselmer}. In \Cref{sec:twists-of-ab-vars}, we apply the results of \Cref{sec:selmer-groups-and-integrality} to prove Theorems \ref{thm:completely-reducible} and \ref{thm:explicit-rank-bound-intro}. 

The remainder of the paper is devoted to applications of our main results. In \Cref{sec:trigonal}, we study the average ranks of the Jacobians of the curves $y^3 = xf(x^{3^{m-1}})$ and $y^{3^m} = f(x)$ and prove \Cref{cor:trigonal}. In \Cref{sec:cm}, we give explicit results for abelian varieties with CM and prove \Cref{thm:cmintro}. In \Cref{sec:uniform}, we study rational points in twist families of curves as in \Cref{subsec:qmintro}, and prove Theorems \ref{thm:uniformcurves} and \ref{thm:uniformtheta}. Finally, in \Cref{sec:qm-surfaces}, we study twist families of genus $3$ curves and prove \Cref{thm:genus3example}.

\section{Abelian varieties with \texorpdfstring{$\zeta$}{zeta}-multiplication}\label{sec:zeta-mult}

Let $F$ be a field of characteristic $0$. Fix an odd prime $p$, an integer $n = p^m$, and a primitive $n$-th root of unity $\zeta = \zeta_n$. In this section, $\varphi$ denotes Euler's totient function.

\begin{definition}
 An abelian variety with $\zeta$-multiplication is a pair $(A, \iota_A)$, where $A$ is an abelian variety over $F$ and $\iota_A\: \Z[\zeta]\hookrightarrow\End_{\bar F} A$ is a $G_F$-equivariant injective ring homomorphism.
\end{definition}

We usually suppress any mention of $\iota_A$ and view $\Z[\zeta]$ as a subring of $\End_{\bar F}A$. In this section, we collect some basic facts and constructions relating to abelian varieties with $\zeta$-multiplication.

\subsection{The isogeny \texorpdfstring{$\pi$}{pi}}
If $A$ is an abelian variety with $\zeta$-multiplication, then since $1-\zeta$ divides $p$ in $\Z[\zeta]$, the map $1 - \zeta \in \End_{\bar F} A$ is an isogeny whose degree is a power of $p$. 

\begin{lemma}\label{lem:descent}
 The kernel of $1- \zeta$ is $G_F$-stable, and hence is an  $F$-subgroup of $A[p]$. In particular, there is an abelian variety $\Aone$ over $F$, such that the endomorphism $1 - \zeta$ of $A$ over $\overline F$ descends to an isogeny $\pi \colon A\to A^{(1)}$ over $F$.
\end{lemma}

\begin{proof}
 If $P\in A_{\bar F}[1-\zeta]$ and $\sigma\in G_F$, then $\zeta^{\sigma\ii} = \zeta^i$ for some $i\in (\Z/n\Z)\t$ and
 \[\zeta(P^\sigma) = (\zeta^{\sigma\ii}(P))^\sigma = (\zeta^i P)^\sigma = P^\sigma,\]
 which shows that $P^\sigma \in A_{\bar F}[1-\zeta]$. Hence, $A_{\bar F}[1-\zeta]$ descends to an $F$-subgroup $H$ of $A$. Thus, we obtain a an isogeny $\pi\:A\to A/H =: \Aone$ over $F$.
 
 The equality of ideals $(1-\zeta)^{\varphi(n)} = (p)$ in $\Z[\zeta]$ shows that $A_{\bar F}[1-\zeta]\sub A_{\bar F}[p]$.
\end{proof}

If $\zeta\in F$, then $\Aone = A/A[1-\zeta] \simeq A$, and $\pi\:A\to \Aone$ can be identified with the endomorphism $1-\zeta$. If $\zeta \notin F$, then $\Aone = A/A[\pi]$ is a twist of $A$ which we now identify:

\begin{lemma}\label{lem:zeta-twist}
$\Aone$ is the twist of $A$ corresponding to the cocycle $\sigma\mapsto \frac{1-\zeta^\sigma}{1-\zeta}\in H^1(F, \Z[\zeta]^\times)$. 
\end{lemma}

\begin{proof}
Over $F(\zeta)$, the map $\eta\colon A/A[1-\zeta]\to A$ given by $\overline x\mapsto (1-\zeta)x$ defines an isomorphism. Hence, $\Aone$ is the twist corresponding to the cocycle $\sigma\mapsto \eta{}^{\sigma}\eta^{-1}$.
\end{proof}

The abelian variety $\Aone$ also has $\zeta$-multiplication. Iterating \Cref{lem:descent}, for each integer $s$, we obtain an abelian variety $\As= A/A[(1-\zeta)^s]$. As in \Cref{lem:zeta-twist}, $\As$ is isomorphic to the twist of $A$ corresponding to the cocycle 
\[\sigma\mapsto \frac{(1-\zeta^\sigma)^s}{(1-\zeta)^s}\in H^1(F, \Z[\zeta]^\times).\] We define $\pi^s$ to be the corresponding isogeny $A = A^{(0)}\to \As$.

Note that $A^{(\varphi(n))} \simeq A$ and that $\pi^{\varphi(n)}\:A\to A$ is multiplication by $pu$, for some unit $u\in\Aut A$. In particular, when writing $\As$, we can always consider $s$ modulo $\varphi(n) = p^{m-1}(p-1)$, and we have inclusions
\[A[\pi]\sub A[\pi^2]\sub\cdots \sub A[\pi^{\varphi(n)-1}]\sub A[\pi^{\varphi(n)}] = A[p].\]

In general, the Galois action on $\As$ is related to the Galois action on $A$ in a convoluted way. However, on a subset of the torsion of $A$, the action is especially simple.

\begin{lemma}\label{lem:tensor-cyclotomic}
 Let $s = p^r$, for some $0\le r< m$, and let $i\in\Z$. Then, as $G_F$-modules,
 \[A^{(is)}[\pi^s] \simeq A[\pi^s]\tensor\chi_p^{i}, \]
 where $\chi_p$ is the mod $p$ cyclotomic character.
\end{lemma}

\begin{proof}
 By \Cref{lem:zeta-twist}, the abelian variety $A^{(is)}$ is the twist of $A$ corresponding to the cocycle $\sigma\mapsto \frac{(1-\zeta^\sigma)^{is}}{(1-\zeta)^{is}}\in H^1(F, \Z[\zeta]^\times)$. Equivalently, there is an isomorphism $\phi\:A\to A^{(is)}$ over $F(\zeta)$ such that for all $\sigma\in G_F$ and $P\in A(\overline F)$, we have
    \[(\phi\ii)^\sigma \circ \phi(P) = \frac{(1-\zeta^\sigma)^{is}}{(1-\zeta)^{is}}P.\]
 
 Hence, if $P\in A[\pi^s]$ and $\sigma \in G_F$, then in $A^{(is)}[\pi^s]$, we have
 \begin{align*}
  (\phi(P))^\sigma = \frac{(1-\zeta^\sigma)^{is}}{(1-\zeta)^{is}} \phi(P^\sigma).
 \end{align*}
 
 Suppose that $\sigma\:\zeta\mapsto \zeta^j$ for some $j\in (\Z/n\Z)\t$. Then, since $ (1-\zeta^{p^r})\phi(P^\sigma) =(1-\zeta)^{p^r}\phi(P^\sigma) = 0$, we have
 \begin{align*}
  \frac{(1-\zeta^\sigma)^{ip^r}}{(1-\zeta)^{ip^r}} \phi(P^\sigma) &=(1+ \zeta + \zeta^2 + \cdots + \zeta^{s -1})^{ip^r}\phi(P^\sigma) \\
  &=(1+ \zeta^{p^r} + \zeta^{2p^r} + \cdots + \zeta^{p^r(j -1)})^{i}\phi(P^\sigma)\\
 &= j^i \phi(P^\sigma).
 \end{align*}
 Since, by definition, $j\pmod p = \chi_p(\sigma)$, we see that $(\phi(P))^\sigma=\chi_p(\sigma)^i \phi(P^\sigma)$, as claimed.
\end{proof}

\subsection{\texorpdfstring{$\zeta$}{zeta}-linear isogenies}\label{sect: zeta-linear-isogenies} We keep the notation $n = p^m$ and $\zeta = \zeta_n$. 

\begin{definition}
Let $(A,\iota_A)$ and $(B, \iota_B)$ be abelian varieties over $F$ with $\zeta$-multiplication, and let $\phi \colon A \to B$ be an isogeny. We say that $\phi$ is \emph{$\zeta$-linear} if $\iota_B(\alpha)\circ \phi = \phi\circ \iota_A(\alpha)$ for all $\alpha \in \Z[\zeta]$.
\end{definition}

\begin{lemma}\label{lem:zeta-linear}
If $\phi \colon A \to B$ is a $\zeta$-linear $p$-isogeny, then $A[\phi] \subset A[\pi]$. Conversely, if $H \subset A[\pi]$ is a $G_F$-stable subgroup of order $p$, then the quotient $B:= A/H$ inherits a $\zeta$-multiplication from $A$, and the canonical $p$-isogeny $\phi \colon A \to B$ is $\zeta$-linear. 
\end{lemma}
\begin{proof}
Since $\phi$ is $\zeta$-linear, if $P\in A[\phi]$, then so is $\zeta P$. Hence, the action of $\zeta$ is given by a homomorphism $\mu_n = \mu_{p^m} \to \Aut_{\bar F} A[\phi] \simeq (\Z/p\Z)^\times$, which must of course be trivial. Thus, $\zeta$ acts as the identity on $A[\phi]$, so $A[\phi]\sub A[1-\zeta] = A[\pi]$.

For the converse, since $\iota_A(\zeta)$ fixes $H$, we have $\ker(\phi) = \ker(\phi \circ \iota_A(\zeta))$, so that $\phi\circ \iota_A(\zeta)$ factors through $\phi$. That is, there exists an automorphism $\zeta_B\: B\to B$ such that $\phi\circ \iota_A(\zeta) = \zeta_B\circ\phi$. This automorphism $\zeta_B$ has order $n$ and has the same minimal polynomial as $\iota_A(\zeta)$. Thus, the map $\iota_B\:\Z[\zeta]\to \End_{\bar F}B$ given by $\zeta\mapsto\zeta_B$ is a $\zeta$-multiplication on $B$, and the $p$-isogeny $\phi$ is $\zeta$-linear by construction.
\end{proof}

\subsection{Twists}\label{subsec:twists}

If $(A, \iota_A)$ has $\zeta_n$-multiplication, then $\iota_A$ induces an inclusion $\Z[\zeta_{2n}]\t \subset \Aut_{\bar F}A$ of $G_F$-modules. For each $d \in F^\times$, let $A_d$ be the twist of $A$ corresponding to the image of $d$ under 
\[F^\times \to F^\times/F^{\times2n} \simeq H^1(F, \mu_{2n}) \to H^1(F, \Aut_{\bar F}A).\] 
Then $A_d$ is an abelian variety over $F$ that becomes isomorphic to $A$ over $F(d^{1/2n})$. Moreover, $A_d$ also has $\zeta_n$-multiplication. 

\begin{remark}
 If $\Aut_{\bar F}A = \mu_{2n} $, then distinct $2n$-th-power classes $d$ give non-isomorphic $A_d$. However, if $\Aut_{\bar F}A \supsetneq \mu_{2n}$, then the map $H^1(F, \mu_{2n}) \to H^1(F, \Aut_{\bar F}A)$ need not be injective, and hence the twists $A_d$ need not be distinct. 
\end{remark}

Now let $\phi \colon A \to B$ be a $\zeta$-linear $p$-isogeny over $F$. By \Cref{lem:zeta-linear}, the automorphisms $\pm\zeta \in \Aut_{\bar F}A$ preserve the subgroup $A[\phi]$, giving an inclusion of $G_F$-modules $\mu_{2n} \hookrightarrow \Aut_{\bar F}( \phi)$, where $\Aut_{\bar F}(\phi)$ is the subgroup of $\Aut_{\bar F}A$ stabilizing $A[\phi]$. For $d \in F^\times$, let $\phi_d \colon A_d \to B_d$ be the twist of $\phi$ corresponding to the image of $d$ under 
\[F^\times \to F^\times/F^{\times2n} \simeq H^1(F, \mu_{2n}) \to H^1(F, \Aut_{\bar F}(\phi)).\] 
Then $\phi_d$ is a $\zeta$-linear $p$-isogeny over $F$. Similarly, we may twist the isogeny $\pi \colon A \to A^{(1)}$ to obtain $\pi_d \colon A_d \to A^{(1)}_d$, and this is the canonical isogeny ``$\pi$'' associated to $A_d$ (and its $\zeta$-multiplication). 

\begin{remark}\label{rem:zeta-twist-vs-d-twist}
	By \Cref{lem:zeta-twist}, the abelian variety $\Aone$ is the twist of $A$ corresponding to the cocycle $\xi\:\sigma\mapsto \frac{1-\zeta^\sigma}{1-\zeta}$ in $H^1(F, \Z[\zeta_n]^\times)$. When $n = 3$, we have $\Z[\zeta_{2n}]^\times = \mu_6$, and since $\frac{1-\zeta}{\sqrt[6]{-27}} \in \mu_6$, the cocycle $\xi$ is in the same cohomology class as $-27\in F^{\times}/F^{\times6}$. It follows that $\Aone = A_{-27}$, which is the quadratic twist of $A$ by the mod $3$ cyclotomic character. More generally, we have $A^{(\frac12\varphi(n))} = A_{(p^*)^n}$, where $p^*=(-1)^{(p-1)/2}p$ is such that $\Q(\sqrt{p^*})\sub \Q(\zeta_p)$. 
	However, for general $s$ and $n$, the twist $\As$ need not be isomorphic to $A_d$, for any $d \in F^\times/F^{\times2n}$.
\end{remark}

\section{Integral orbits of binary cubic forms}\label{sect: integral}
In this section, we classify the integral $\SL_2(F)$-orbits on the space of binary cubic forms over a local field $F$. We first recall some facts from \cite[\S 2]{elk} and \cite[Thm.\ 13]{HCL1}.

Let $V = \Sym^3\Z^2$ be the space of binary cubic forms. The group $\SL_2$ acts on $V$, and the ring of invariant functions is generated by the usual polynomial discriminant $\Disc \colon V \to \Z$. Let $F$ be any field of characteristic not 2 or 3, and for any $d \neq 0$ in $F$, define $V(F)_d := \{f \in V(F) \colon \Disc(f) = d\}$. There is a unique \emph{reducible} $\SL_2(F)$-orbit of cubic forms $f \in V(F)_d$. The stabilizer of such an $f$ is a commutative $F$-group scheme $C_d$ of order $3$. The Galois action on $C_d(\overline F)$ is by the quadratic character $\chi_d \colon \Gal(F(\sqrt d)/F) \to \{\pm 1\}$. 
\begin{proposition}\label{prop:AIT}
The group $H^1(F, C_d)$ is in bijection with the $\SL_2(F)$-orbits on $V(F)_d$.
\end{proposition}

Now let $F$ be a local field of residue characteristic neither $2$ nor $3$, with surjective discrete valuation $v \colon F^\times \to \Z$, ring of integers $\O_F$, maximal ideal $\m$ and residue field $\F$ of cardinality $q$. 

We wish to determine which $\SL_2(F)$-orbits in $V(F)_d$ have representatives in $V(\O_F)_d$. We call these orbits the \emph{integral orbits}, and we let $H^1_\mathrm{int}(C_d)$ be the subset of $H^1(F,C_d)$ that they correspond to under the bijection of \Cref{prop:AIT}. Of course, a necessary condition for there to be any integral orbits at all is that $d \in \O_F$. We see that even though the abstract group $H^1(F,C_d)$ depends only on the square-class of $d$, the notion of integrality depends on the actual value of $d$, and in particular its valuation. 

We recall some facts about cubic rings over $F$ and over $\O_F$ \cite{BST}. The action of $\SL_2$ on $V$ extends to the following action of $\GL_2$: if $\gamma = (\begin{smallmatrix} a & b\\c& d\end{smallmatrix})$, then 
\[ (\gamma\cdot f)(x,y) = \frac{1}{\det \gamma} f(ax+cy,bx+dy).\] 
If $R$ is a principal ideal domain, then a \emph{cubic ring over $R$} is an $R$-algebra $S$ that is free of rank $3$ as an $R$-module. The discriminant of $S$ is a well-defined element of $R^\times/R^{\times 2}$.

\begin{proposition}[Levi, Delone--Faddeev, Gan--Gross--Savin]\label{prop:DeloneFaddeev}
 For any principal ideal domain $R$, there is a discriminant preserving bijection between $\GL_2(R)$-orbits on $V(R)$ and isomorphism classes of cubic rings over $R$. Moreover, this bijection is functorial in $R$.
\end{proposition}
\begin{proof}
Building on \cites{levi, delone-faddeev, gan-gross-savin}, it is shown in \cite{BST} that the bijection sends a cubic $R$-ring $S$ to any binary cubic form representing the cubic map $S/R \to \wedge^2_R (S/R)$, $s \mapsto s \wedge s^2$, which is functorial in $R$. 
\end{proof}

If $\gamma \in \GL_2(F)$ and $f \in V(F)$, then $\Disc(\gamma f) = \det(\gamma)^2\Disc(f).$ It follows that isomorphism classes of cubic $F$-algebras $L$ of discriminant $d$ are in bijection with $\GL_2(F)_{\pm1}$-orbits on $V(F)_d$. Here, $\GL_2(F)_{\pm1}$ is the subgroup of $\GL_2(F)$ consisting of elements with determinant $\pm 1$. 
Since $\SL_2(F)$ has index two in $\GL_2(F)_{\pm1}$, the $\GL_2(F)_{\pm1}$-orbits break up into at most two $\SL_2(F)$-orbits. It is a fun exercise to show that there are exactly two orbits if and only if $L$ is a field; the orbits are represented by $f(x,y)$ and $f(y,x)$.

\begin{remark}\label{rem:trivial-integral}
    The trivial class in $H^1(F, C_d)$ corresponds to the unique orbit of \emph{reducible} forms of discriminant $d$. Hence, $\alpha \in H^1(F, C_d)$ is non-trivial if and only if the corresponding cubic algebra $L$ is a field (if and only if $L$ is generated over $F$ by a root of $f(x,1)$). The trivial class corresponds to $F \times E_d$, where $E_d = F[x]/(x^2 - d)$ is the quadratic $F$-algebra of discriminant $d$.  Note that the trivial class is represented by $\frac d4x^3 + xy^2$, which is integral as long as $d$ is.
\end{remark}

From the functoriality in \Cref{prop:DeloneFaddeev} applied to the base change $\O_F \hookrightarrow F$, we deduce:
\begin{proposition}\label{prop:integrality criterion}
Let $\alpha \in H^1(F, C_d)$, and let $L$ be the corresponding cubic $F$-algebra. Then $\alpha$ is integral if and only if there is an $\O_F$-order $S \subset \O_L$ with $v(\Disc \, S) = v(d)$. 
\end{proposition}
\begin{proof}
    The ``only if'' direction is clear. For the ``if'' direction, observe that any $\O_F$-order $S \subset \O_L$ has discriminant congruent to $d$ modulo $F^{\times 2}$. So if $v(\Disc\, S) = v(d)$, then we see that $\Disc(S)$ is congruent to $d$ modulo $\O_F^{\times 2}$. Thus we may choose bases so that $S$ corresponds to a binary cubic form with coefficients in $\O_F$ of exact discriminant $d$.  
\end{proof}

The following two facts about cubic orders will be useful \cite[Props.\ 15-16]{BST}.
\begin{proposition}\label{splitting}
Let $L$ be an \'etale cubic $F$-algebra. Suppose $f(x,y)$ corresponds to the maximal order $\O_L$ under the bijection of Proposition $\ref{prop:DeloneFaddeev}$. Then the factorization type of $f(x,y)$ over the residue field $\F$ is the factorization type of the maximal ideal $\m$ of $\O_F$ in the ring $\O_L$. 	
\end{proposition}
\begin{proposition}\label{order}
Let $f(x,y) \in V(\O_F)_d$ correspond to a cubic ring $S$ over $\O_F$. Then the sub-$\O_F$-rings $S' \subset S$ of index $q$ correspond bijectively with the zeros of $f \pmod \m$ in $\P^1(\F)$. 
\end{proposition}

We also need the following result, which requires $\charr \F \neq 3$, and which describes the  subgroup $H^1_{\mathrm{un}}(F, C_d) \subset H^1(F, C_d)$ of unramified classes. 
\begin{proposition}\label{prop:unram}
Suppose $0 \neq \alpha \in H^1(F, C_d)$ corresponds to the cubic extension $L/F$. Then $\alpha \in H^1_\un(F,C_d)$ if and only if $L$ is unramified.
\end{proposition}

\begin{proof}
Let $f_L \in V(F)$ be the corresponding binary cubic form. If $L$ is unramified, then since $f_L$ becomes reducible over $L$, the restriction of $\alpha$ to $H^1(L, C_d)$ is trivial. Thus, $\alpha$ is an unramified class. If $L$ is ramified, then $f_L$ remains irreducible over every unramified extension of $F$, and hence $\alpha$ is ramified. 
\end{proof}

\begin{lemma}\label{lem:table-h10cd}
Assume the residue characteristic of $F$ is not $3$. Then:
\begin{enumerate}
    \item $\dim H^1(F, C_d) = \dim H^0(F, C_d) + \dim H^0(F, C_{-3d})$.
    \item $\dim H^1_\un(F,C_d) = \dim H^0(F, C_d)$.
\end{enumerate} 
When $d$ has even valuation, these dimensions are computed in Table $\ref{norms-dimension-table_cd}$ below.
\end{lemma}

\begin{table}[h]
	\caption{Dimensions of $H^1(F, C_d)$ and $H^1_\un(F, C_d)$}\label{norms-dimension-table_cd}

\begin{tabular}{|c|c|c|}
	\hline
&$d\in F^{\times2}$&$-3d\in F^{\times2}$\\
\hline
$\zeta_3\in F$&\multicolumn{2}{c|}{\begin{tabular}{l}
	$\dim H^1(F, C_d)=2$\\
	$\dim H^1_\un(F, C_d)=1$\end{tabular}}\\
\hline
$\zeta_3\notin F$&\begin{tabular}{l}
	$\dim H^1(F, C_d)=1$\\
	$\dim H^1_\un(F, C_d)=1$
\end{tabular}&\begin{tabular}{l}
$\dim H^1(F, C_d)=1$\\
$\dim H^1_\un(F, C_d)=0$
\end{tabular}\\
\hline
\end{tabular}
\end{table}
\begin{proof}
First note that $C_d$ is Cartier dual to $C_{-3d} \simeq C_d \otimes \mu_3$. Since the residue characteristic is not $3$, the Euler--Poincar\'e characteristic formula \cite[I.2.8]{Milne} immediately gives $(i)$. Let $I_F \subset G_F$ be the inertia group and let $g$ be the Frobenius element of $G_F/I_F$. Then the groups $H^1_\un(F, C_d) \simeq H^0(I_F, C_d)/(g-1)H^0(I_F,C_d)$ and $H^0(I_F,C_d)[g-1] = H^0(F, C_d)$ have the same cardinality, which proves $(ii)$. The table is computed using the fact that $\dim H^0(F,C_d) = 1$ if and only if $d \in F^{\times 2}$ and the dimension is $0$ otherwise.
\end{proof}

 The main result of this section is the following classification of the integral orbits in $V(F)_d$.  
\begin{theorem}\label{thm:integralorbits}
Let $\O_F$ be the ring of integers of a local field $F$ with $\charr \O_F/\m > 3$, and let $d \in \O_F$ be non-zero.
\begin{enumerate}[label=$(\alph*)$]
\item If $v(d) = 0$, then $H^1_\mathrm{int}(F, C_d) = H^1_\un(F, C_d)$. 
\item If $v(d)$ is odd, then $H^1_\mathrm{int}(F, C_d) = H^1 (F, C_d)=0$. 
\item If $v(d) = 2$, then the only non-integral classes are the non-trivial unramified classes. 
\item If $v(d) > 2$, then all classes are integral.  
\end{enumerate}
\end{theorem} 
\begin{proof}
\begin{enumerate}[label=$(\alph*)$]
\item This case follows from Propositions \ref{prop:integrality criterion} and \ref{prop:unram}.
\item 
We have $H^1(F,C_d) = 0$ by Lemma \ref{lem:table-h10cd} (since $H^0(F,C_d) = 0$ whenever $d$ has odd valuation).
By \Cref{rem:trivial-integral}, the trivial class is integral.
\item The ramified classes $\alpha$ correspond to totally ramified cubic extensions $L/F$. For such $L$ we have $v(\Disc \O_L) = v(d)$, and hence these $\alpha$ are integral by \Cref{prop:integrality criterion}. If $H^1(F,C_d)$ has a non-trivial unramified class $\alpha$, then it corresponds to the unique unramified cubic extension $L/F$, which has unit discriminant. By \Cref{prop:integrality criterion}, $\alpha$ is integral if and only if $\O_L$ has an order of index $q$. By \Cref{splitting}, the binary form corresponding to $\O_L$ has no root over $\F$. So by \Cref{order}, $\O_L$ has no order of index $q$. Hence the non-trivial unramified classes are indeed non-integral. 
\item If $\alpha \in H^1(F,C_d)$ corresponds to a ramified cubic extension $L/F$, then $\O_L$ has discriminant of valuation $2$. By Propositions \ref{splitting} and \ref{order}, $\O_L$ has a unique order $S$ of index $q$, and hence $v(\Disc \, S) = 4$. Note that if $S_0$ is a cubic $\O_F$-ring, then $S_0' = \O_F + \m S_0$ is a subring of $S$ of index $q^2$ and $\Disc(S'_0) = q^4\Disc(S_0)$. Thus, by considering the orders $\O_F + \m^k \O_L$ and $\O_F + \m^k S$, we can find an order $S' \subset \O_L$ with $v(\Disc \, S') = 2k$, for any $k \geq 1$.

Next let $\alpha \in H^1(F,C_d)$ be a non-trivial unramified class corresponding to the unramified cubic extension $L/F$. Then $v(\Disc(\O_F + \m\O_L)) = 4$. By \Cref{order}, there are $q + 1$ suborders $S' \subset \O_F + \m \O_L$ of index $q$, so that $v(\Disc \, S') = 6$. As  before, we deduce that there exists an order $S'' \subset \O_L$ with $v(\Disc \, S'') = 2k$, for any $k \geq 2$.   
\end{enumerate}
\end{proof}

\section{Local Selmer conditions for \texorpdfstring{$\zeta$}{zeta}-linear isogenies}\label{sec:localselmer}

Let $F$ be a finite extension of $\Q_p$ with surjective discrete valuation $v$, ring of integers $\O_F$, uniformizer $\varpi$, and residue field $\F$. Let $m \geq 1$, $n = 3^m$, and let $\zeta = \zeta_n$ be a primitive $n$-th root of unity. Let $A$, $B$ be abelian varieties over $F$ that admit $\zeta$-multiplication.

Let $\phi \colon A \to B$ be a $\zeta$-linear $3$-isogeny over $F$, as defined in \Cref{sect: zeta-linear-isogenies}. For each $d\in F\t$, we consider the $3$-isogeny $\phi_d \colon A_d \to B_d$, as in \Cref{subsec:twists}. We will assume in this section that $A[\phi](F) \neq 0$, that is, that $A[\phi]$ is generated by a rational point. This can always be achieved by replacing $\phi$ with an appropriate twist, so there is no loss in generality. We also assume that $0 \leq v(d) < 2n$, again with no loss in generality.

The group $H^1(F, A_d[\phi_d])$ is a finite dimensional $\F_3$-vector space. In fact, if $\chi_d \colon G_F \to \F_3^\times$ is the quadratic character cutting out $F(\sqrt{d})$, then $A_d[\phi_d] \simeq A[\phi] \otimes \chi_d$ is isomorphic to $C_d$ from \Cref{sect: integral}. Thus, $H^1(F, A_d[\phi_d]) \simeq H^1(F,C_d)$. The exact sequence 
\[0 \to A_d[\phi_d] \to A_d \to B_d \to 0\]
induces a Kummer map
\[\partial_d \colon B_d(F) \to H^1(F,A_d[\phi_d]).\]
We call its image $\im(\partial_d) \subset H^1(F,A_d[\phi_d])$ the subgroup of \emph{soluble classes}. The goal of this section is to prove the following theorem describing the soluble classes and to compute the local Selmer ratios of $\phi_d$.

\begin{theorem}\label{thm:soluble orbits}
Assume that $\charr\F\neq3$, that $A$ has good reduction, and that $A[\phi](F)\ne 0$. Suppose that $0 \leq v(d) < 2n$.
\begin{enumerate}[label=$(\alph*)$]
\item If $v(d) = 0$, then $\im(\partial_d) = H^1_{\mathrm{un}}(F, A_d[\phi_d])$. 
\item If $v(d)$ is odd, or more generally, if $F(\sqrt{d})/F$ is ramified, then $\im(\partial_d) = H^1(F, A_d[\phi_d]) = 0$. 
\item If $v(d) > 0$ is even and $d\notin F^{\times 2}$, then $H^1_{\mathrm{un}}(F, A_d[\phi_d]) = 0$.
\item If $v(d) > 0$ is even and $d \in F^{\times 2}$, then write $d = \varpi^{v(d)}u$ with $u\in\O_F\t$ and let $s = \gcd(3^m,v(d))$.
Then $\im(\partial_d) \cap H^1_{\mathrm{un}}(F, A_d[\phi_d]) = 0$ if and only if $A_u[\phi]$ is a direct summand of $A_u[\pi^{s}]$.
\end{enumerate}
\end{theorem} 

We retain these assumptions on $A$ and $\F$ for the remainder of this section.

\subsection{Non-square and unramified twists}

We first prove parts $(a)$-$(c)$.

\begin{proof}[Proof of Theorem $\ref{thm:soluble orbits}(a)\text{-}(c)$]
For $(a)$, assume at first that $\charr \F > 3$. Then $F(d^{1/2n})$ is unramified over $F$, since $v(d) = 0$ and $(\charr \F, 2n) = 1$. Since $A_d$ is isomorphic to $A$ over $F(d^{1/2n})$, it has good reduction over an unramified extension, and hence has good reduction already over $F$. Since $A_d$ has good reduction and $\charr \F \nmid \deg(\phi_d)$, the image of the Kummer map $\partial_d$ is exactly the unramified classes \cite[Prop.\ 2.7$(d)$]{CesnaviciusFlat}. The proof just given works even when $\charr \F = 2$, as long as $F(\sqrt{d})/F$ is unramified. When this extension is ramified, the result follows from $(b)$. Part $(b)$ itself follows from \Cref{thm:integralorbits}$(b)$. The case $\charr \F = 2$ was not dealt with there, but the proof is identical.

For $(c)$, we have $H^1_\un(F, A_d[\phi_d]) = H^1_\un(F, C_d) =  0 $ by \Cref{lem:table-h10cd}, since $H^0(F, C_d) = 0$ whenever $d\notin F^{\times 2}$.
\end{proof}

\subsection{Twists of positive valuation}

The proof of \Cref{thm:soluble orbits}$(d)$ will take more work. Indeed, we will compute more generally the size of $\im(\partial_d)$ for all $d$ (including $d\notin F^{\times 2}$) such that $v(d)$ is even and positive. Let $s = \gcd(v(d), 3^m)$, and write $d = \varpi^{v(d)}u$ for $u\in \O_F\t$. Recall the map $\pi^s\:A\to \As$ defined in \Cref{sec:zeta-mult}, and let $\psi^s\:B\to \As$ be the isogeny such that $\psi^s\circ\phi = \pi^s$.
\subsubsection{Extension classes}
For each $t \in F^\times$, let $\kappa^s_t$ be the extension class corresponding to the short exact sequence
\begin{equation}\label{eq:short-ex-seq-1}0 \to A_t[\phi_t] \to A_t[\pi^s_t] \xrightarrow{\phi_t} B_t[\psi^s_t] \to 0.\end{equation}
Thus $\kappa_t^s = 0$ if and only if $A_t[\phi_t]$ is a direct summand of $A_t[\pi_t^s]$ as a $G_F$-module. Similarly, let $\widehat\kappa_t^s$ be the class of the extension
\begin{equation}\label{eq:short-ex-seq-2}0\to B_t[\psi^s_t]\to B_t[\pi_t^s]\xrightarrow{\psi^s_t} \As_t[\phi^{(s)}_t]\to 0.\end{equation}
Thus, $\widehat\kappa_t^s = 0$ if and only if $B_t[\psi^s_t]$ is a direct summand of $B_t[\pi_t^s]$.

\begin{remark}\label{rem:duality}
 By \Cref{lem:tensor-cyclotomic}, the cocycle $\widehat\kappa_t^s$ is equal to the class of the extension
 \[0\to B_t^{(-s)}[\psi_t^{-s}]\to B_t^{(-s)}[\pi_t^s]\to A_t[\phi_t]\to 0.\]
 Here, we have $\phi\circ\psi_t^{-s} = \pi^s\:B_t^{(-s)}\to B_t$. By duality, $\widehat\kappa_t^s$ is the class of the extension
 \[0\to \widehat B_t[\widehat\phi_t] \to \widehat B_t[\widehat\pi_t^{s}] \to \widehat A_t[\widehat\psi_t^{-s}] \to 0.\]
 Thus $\widehat\kappa^s_t = 0$ if and only if $\widehat B_t[\widehat\phi_t]$ is a direct summand of $\widehat B_t[\widehat\pi_t^s]$, which explains the notation.
\end{remark}

Let $|\kappa_u^s|$ and $|\widehat\kappa_u^s|$ denote the orders of the classes $\kappa_u^s$ and $\widehat{\kappa}_u^s$ in their respective Ext-groups. 

\subsubsection{The image of $\partial_d$}

The following theorem relates the size of $\im\partial_d$ to $|\kappa_u^s|$ and $|\widehat\kappa_u^s|$ and will finish the proof of \Cref{thm:soluble orbits}$(d)$.

\begin{theorem}\label{thm:kummer-image}
    Assume that $v(d)$ is even and positive, and write $d = \varpi^{v(d)}u$ with $u\in \O_F\t$.
    \begin{enumerate}
        \item $\#\im\partial_d\cap H^1_\un(F, A_d[\phi_d]) =\begin{cases} |\kappa_u^s|&d\in F^{\times 2}\\1&\text{otherwise.}\end{cases}$
        \item $\#\br{\frac{\im\partial_d}{\im\partial_d\cap H^1_\un(F, A_d[\phi_d])}}=\begin{cases} \frac{3}{|\widehat\kappa_u^s|}&-3d\in F^{\times 2}\\1&\text{otherwise.}\end{cases}$
    \end{enumerate}
    In particular, if $d\in F^{\times 2}$,  then $\im\partial_d\cap H^1_\un(F, A_d[\phi_d])=0$ if and only if $A_u[\phi_u]$ is a direct summand of $A_u[\pi^s_u]$.
\end{theorem}

The proof of \Cref{thm:kummer-image} requires several preliminary results.
\begin{lemma}
The dimensions of $H^1(F, A_d[\phi_d])$ and $H^1_\un(F,A_d[\phi_d])$ are as in Table $\ref{norms-dimension-table}$.
\end{lemma}

\begin{table}[h]
	\caption{Dimensions of $H^1(F, A_d[\phi_d])$ and $H^1_\un(F, A_d[\phi_d])$}\label{norms-dimension-table}

\begin{tabular}{|c|c|c|c|}
	\hline
&$d\in F^{\times2}$&$-3d\in F^{\times2}$&$d, -3d\notin F^{\times2}$\\
\hline
$\zeta_3\in F$&\multicolumn{2}{c|}{\begin{tabular}{l}
	$\dim H^1(F, A_d[\phi_d])=2$\\
	$\dim H^1_\un(F, A_d[\phi_d])=1$\end{tabular}}&\begin{tabular}{l}
$\dim H^1(F, A_d[\phi_d])=0$\\
$\dim H^1_\un(F, A_d[\phi_d])=0$
\end{tabular}\\
\hline
$\zeta_3\notin F$&\begin{tabular}{l}
	$\dim H^1(F, A_d[\phi_d])=1$\\
	$\dim H^1_\un(F, A_d[\phi_d])=1$
\end{tabular}&\begin{tabular}{l}
$\dim H^1(F, A_d[\phi_d])=1$\\
$\dim H^1_\un(F, A_d[\phi_d])=0$
\end{tabular}&\begin{tabular}{l}
$\dim H^1(F, A_d[\phi_d])=0$\\
$\dim H^1_\un(F, A_d[\phi_d])=0$\end{tabular}\\
\hline
\end{tabular}
\end{table}
\begin{proof}
Since $H^1(F, A_d[\phi_d]) \simeq H^1(F, C_d)$ and $H^1_\un(F,A_d[\phi_d]) \simeq H^1_\un(F,C_d)$, the dimensions in \Cref{norms-dimension-table} are identical to those in \Cref{norms-dimension-table_cd}.
The bottom right cell is only relevant when $\charr \F = 2$, and follows from Lemma \ref{lem:table-h10cd}(i). \end{proof}

\begin{lemma}\label{lem:units}
If $t\in F^{\times2 s}$, then $A_t[\pi_t^{s}]\simeq A[\pi^{s}]$ and $B_t[\pi_t^{s}]\simeq B[\pi^{s}]$ as $G_F$-modules. 
\end{lemma}

\begin{proof}
By the definition of $A_t$, there is an isomorphism $\phi\:A\to A_t$ over $F(t^{1/2n})$ such that if $P\in A[\pi^{s}]$ and $\sigma\in G_F$, then in $A_t[\pi_t^s]$, we have
\[(\phi(P))^\sigma = \frac{\sigma(t^{1/2n})}{t^{1/2n}}\phi(P^\sigma). \]

	If $t\in F^{\times 2 s}$, then $\frac{\sigma(t^{1/2n})}{t^{1/2n}}\in \langle \zeta^{s}\rangle$. Since $\zeta^{s}$ acts as the identity on $A[\pi^{s}]$, we see that $A_t[\pi_t^{s}] \simeq A[\pi^{s}]$. The proof that $B_t[\pi_t^{s}]\simeq B[\pi^{s}]$ is identical.
 	\end{proof}

\begin{remark}\label{rem:units}
    In particular, if $s=1$, the condition in Theorem $\ref{thm:soluble orbits}(d)$ is simply that $A[\phi]$ is a direct summand of $A[\pi]$, which is independent of $u$.
\end{remark}

\begin{corollary}\label{cor:splitting}\mbox{}
\begin{enumerate}
    \item There is an isomorphism between
    \[0 \to A_u[\phi_u] \to A_u[\pi^s_u] \xrightarrow{\phi_u} B_u[\psi^s_u] \to 0\]
    and 
    \[0 \to A_d[\phi_d] \to A_d[\pi^s_d] \xrightarrow{\phi_d} B_d[\psi^s_d] \to 0\]
    as short exact sequences of $G_F$-modules.
    \item There is an isomorphism between
    \[0\to B_u[\psi^s_u]\to B_u[\pi_u^s]\xrightarrow{\psi^s_u} \As_u[\phi^{(s)}_u]\to 0\]
    and 
    \[0\to B_d[\psi^s_d]\to B_d[\pi_d^s]\xrightarrow{\psi^s_d} \As_d[\phi^{(s)}_d]\to 0\]
    as short exact sequences of $G_F$-modules.
\end{enumerate}
\end{corollary}

\begin{proof}
    The first claim follows from \Cref{lem:units} together with the observation that the isomorphism $A_u[\pi^s_u]\to A_d[\pi^s_d]$ restricts to an isomorphism $A_u[\phi_u]\to A_d[\phi_d]$. The second claim follows similarly.
\end{proof}

\begin{lemma}\label{lem:torsion}
Suppose that $v(d) = 2a\cdot 3^r$ is even and positive, and let $k$ be an unramified extension of $F$. For $X\in\{A, B\}$, we have
$X_d[3](k) = X_d[\pi_d^{3^r}](k).$
\end{lemma}

\begin{proof}
Let $F_\un$ be the maximal unramified extension of $F$, let $L = F_\un(\sqrt{d})$, and let $M = L(d^{1/n}) = F_\un(d^{1/2n})$. Since $k\sub L$, it suffices to show that $X_d[3](L) = X_d[\pi_d^{3^r}](L)$.

The extension $M/L$ is tamely ramified of order $3^{n-r}$, and we can choose a generator $\tau$ of $\Gal(M/L)$ so that $\tau(d^{1/2n})/d^{1/2n} = \zeta^{3^r}$. Since $X$ has good reduction, we have $X[3] \sub X(L)$.  Since $X$ and $X_d$ are isomorphic over $L(d^{1/n})$, it follows that $X_d[3]\sub X(M)$. 

Now, if $\phi\:X_d\to X$ is an isomorphism over $M$ and $P\in X_d(M)$, then by definition,
\[\phi(P)^\tau =  \tau(d^{1/2n})/d^{1/2n}\phi(P^\tau) = \zeta^{3^r}\phi(P^\tau). \]
Hence, if $P\in X_d[3](L)$, then $\phi(P) \in X[3] \sub X[3](L)$, so $\phi(P) = \zeta^{3^r}\phi(P)$. It follows that
\[0 = (1- \zeta^{3^r})\phi(P) = (1-\zeta)^{3^r}\phi(P) = \pi^{3^r}\phi(P) = \phi(\pi_d^{3^r}P),\]
where the second equality uses the fact that $P$ is a $3$-torsion point. Hence $P \in X_d[\pi_d^{3^r}]$, so $X_d[\pi_d^{3^r}](L) = X_d[3](L)$.
\end{proof}

From $(\ref{eq:short-ex-seq-1})$, we obtain a long exact sequence
\[0 \to A_d[\phi_d](F) \to A_d[\pi^s_d](F) \xrightarrow{\phi_d} B_d[\psi^s_d](F)\xrightarrow{\delta_d}H^1(F, A_d[\phi_d]).\]
Similarly, from $(\ref{eq:short-ex-seq-2})$, there is a long exact sequence
\begin{equation}\label{eq:long-exact}
    0\to B_d[\psi^s_d](F)\to B_d[\pi_d^s](F)\xrightarrow{\psi^s_d} \As_d[\phi^{(s)}_d](F)\xrightarrow{\widehat\delta_d}H^1(F, B_d[\psi^s_d]).
\end{equation}
Clearly $\im\delta_d\sub\im\partial_d$, and $\partial_d$ induces an injective map 
\[\frac{B_d[\pi_d^s](F)}{B_d[\psi^s_d](F)}\hookrightarrow\frac{\im\partial_d}{\im\delta_d}.\]

\begin{proposition}\label{prop:im-delta}
    We have
    \[\im\delta_d = \im\partial_d\cap H^1_\un(F, A_d[\phi_d]).\]
\end{proposition}

\begin{proof}
    Consider the Kummer exact sequence
   \[0\to A_d[\phi_d](F)\to A_d(F)\xrightarrow{\phi_d} B_d(F)\xrightarrow{\partial_d}H^1(F, A_d[\phi_d]).\]
    Since $\charr\F\ne 3$, we have $B_d(F)/\phi_dA_d(F) = B_d[3^\infty](F)/\phi_dA_d[3^\infty](F)$, so $\partial_d$ depends only on the exact sequence
    \[0\to A_d[\phi_d](F)\to A_d[3^\infty](F)\xrightarrow{\phi_d} B_d[3^\infty](F)\xrightarrow{\partial_d}H^1(F, A_d[\phi_d]).\]
    By \Cref{lem:torsion}, this exact sequence is the same as
    \[0\to A_d[\phi_d](F)\to A_d[\pi_d^s](F)\xrightarrow{\phi_d} B_d[\pi_d^s](F)\xrightarrow{\partial_d}H^1(F, A_d[\phi_d]).\]
    
    By \Cref{cor:splitting}, the long exact sequence
    \[0 \to A_d[\phi_d](F) \to A_d[\pi^s_d](F) \xrightarrow{\phi_d} B_d[\psi^s_d](F)\xrightarrow{\delta_d}H^1(F, A_d[\phi_d]) \]
    is isomorphic to the long exact sequence
    \[0 \to A_u[\phi_u](F) \to A_u[\pi^s_u](F) \xrightarrow{\phi_u} B_u[\psi^s_u](F)\xrightarrow{\delta_u}H^1(F, A_u[\phi_u]).\]
    Since $A_u$ has good reduction, the image of $\delta_u$ is contained in $H^1_\un(F, A_u[\phi_u])$ \cite[Prop.\ 2.7$(d)$]{CesnaviciusFlat}. Thus the image of $\delta_d$ is contained in both $H^1_\un(F, A_d[\phi_d])$ and $\im\partial_d$.

    Conversely, if $y\in B_d[\pi_d^s](F)\setminus B_d[\psi_d^s](F)$, then $\partial_d(y)$ is the cocycle $\sigma\mapsto x^\sigma - x$, where $\phi_d(x) = y$. But $x\in A_d[3]\setminus A_d[\pi_d^s]$, so by \Cref{lem:torsion}, $x$ cannot be defined over an unramified extension of $F$. It follows that $\partial_d(y)\notin H^1_\un(F, A_d[\phi_d])$. Hence $\im\delta_d\supset  \im\partial_d\cap H^1_\un(F, A_d[\phi_d])$.
    \end{proof}

We next relate the sizes of the images of $\delta_d$ and $\widehat\delta_d$ to the sizes of $|\kappa_u^s|$ and $|\widehat\kappa_u^s|$.

\begin{lemma}\label{lem:extclasses}
If $d\in F^{\times 2}$, then $\#\im\delta_d = |\kappa^s_u|$. Similarly, if $-3d\in F^{\times 2}$, then $\#\im\widehat\delta_d = |\widehat\kappa_u^s|$. 
\end{lemma}
\begin{proof}
By duality, we have $A_u[\phi_u] \simeq \widehat B_u[\widehat\phi_u] \otimes \chi_3$. Thus, the two claims of the lemma are in fact equivalent to each other; see \Cref{rem:duality}. We prove the second one.
Since $-3d\in F^{\times 2}$, we have $\As_u[\phi_u^{(s)}]= \F_3$ as a $G_F$-module. Moreover, by definition we have $\widehat\delta_d(1) = \widehat\kappa_d^s$. By \Cref{cor:splitting}, $|\widehat\kappa_d^s| = |\widehat\kappa_u^s|$. It follows that $|\widehat \kappa_u^s| = \#\im\widehat\delta_d$.
\end{proof}

\begin{proof}[Proof of Theorem $\ref{thm:kummer-image}$]
If $d\in F^{\times 2}$, then part $(i)$ follows immediately from \Cref{prop:im-delta} and \Cref{lem:extclasses}. If $d\notin F^{\times2}$, then $H^1_\un(F, A_d[\phi_d]) = 0$ by \Cref{norms-dimension-table}.

By \Cref{lem:torsion} and \Cref{prop:im-delta}, $\frac{\im\partial_d}{\im\partial_d\cap H^1_\un(F, A_d[\phi_d])}$ is isomorphic to the image of the injective map
\[\frac{B_d[\pi_d^s](F)}{B_d[\psi^s_d](F)}\to \frac{\im\partial_d}{\im\delta_d}\]
induced by $\partial_d$. From $(\ref{eq:long-exact})$, we have
\[\#\br{\frac{B_d[\pi_d^s](F)}{B_d[\psi^s_d](F)}} = \frac{\#\As_d[\phi_d^{(s)}](F)}{\#\im\widehat\delta_d}.\]
If $-3d\in F^{\times2}$, this is $\frac{3}{|\widehat\kappa_u^s|}$ by \Cref{lem:extclasses}. If $-3d\notin F^{\times2}$, then $H^1(F, A_d[\phi_d]) = H^1_\un(F, A_d[\phi_d])$ by \Cref{norms-dimension-table}, so the result follows from \Cref{prop:im-delta}.
\end{proof}

\subsection{Local Selmer ratios}\label{subsec:local-ratios}
For applications, we record the local Selmer ratios 
\[c(\phi_d) = \dfrac{\#\coker\left(A_d(F) \to B_d(F)\right)}{\#\ker(A_d(F) \to B_d(F))} = \dfrac{\#\im(\partial_d)}{\#A_d[\phi_d](F)},\]
which have implicitly been computed in the previous subsection.

\begin{theorem}\label{thm:tamagawa-ratios}
Assume that $\charr \F \neq 3$, that $A$ has good reduction, and that $A[\phi](F) = \Z/3\Z$. Then $c(\phi_d) = 1$ unless $v(d)$ is even and positive. If $v(d)$ is even and positive, write $d = \varpi^{v(d)}u$ with $u\in\O_F\t$, and let $s =\gcd(3^m, v(d))$. Let $\kappa_u^s$ and $ \widehat\kappa_u^s$ be the classes defined in the previous section, and write $|\kappa_u^s|$ and $|\widehat\kappa_u^s|$ for their orders. Then $c(\phi_d)$ is as in Table $\ref{table:selmerratio}$.

\begin{table}[h]
	\caption{Values of $c(\phi_d)$ over local fields $F$, when $v(d)$ is even and positive}\label{table:selmerratio}
\begin{tabular}{|c|c|c|c|}
	\hline
&$d\in F^{\times2}$&$-3d\in F^{\times2}$&$d, -3d\notin F^{\times2}$\\
\hline
\rule{0pt}{3.75ex}
$\zeta_3\in F$&\multicolumn{2}{c|}{\begin{tabular}{l l}
$\dfrac{|\kappa_u^s|}{|\widehat\kappa_u^s|}$
\end{tabular}}& \begin{tabular}{l}
$1$
\end{tabular}\\[1.5ex]
\hline
\rule{0pt}{3.75ex}
$\zeta_3\notin F$&\begin{tabular}{l l}
$\dfrac{|\kappa_u^s|}{3}$
\end{tabular}&\begin{tabular}{l l}
	$\dfrac{3}{|\widehat\kappa_u^s|}$
\end{tabular}&$1$\\[1.5ex]
\hline
\end{tabular}
\end{table}
\end{theorem} 
\begin{proof}
If neither $d$ nor $-3d$ is a square in $F$, then $\#A_d[\phi_d](F) = 1$, and by \Cref{norms-dimension-table}, $\#\im(\partial_d) = 1$, so $c(\phi_d) = 1$ as claimed. Henceforth, we assume that either $d$ or $-3d$ is a square in $F$. When $v(d) = 0$, then $A_d$ has good reduction, so by \cite[Prop.\ 3.1]{shnidmanRM} 
\[c(\phi_d) = c(B_d)/c(A_d) = 1,\]
where $c(A_d)$ and $C(B_d)$ are the Tamagawa numbers of $A_d$ and $B_d$.
When $v(d)$ is odd, then $\im(\partial_d) = 0$ (\Cref{thm:soluble orbits}) and $A_d[\phi_d](F) = 0$, so again $c(\phi_d) = 1$. So it remains to compute $c(\phi_d)$ when $v(d)$ is even and positive. This is done by combining the formula for $\#\im\partial_d$ in   \Cref{thm:kummer-image}  with the fact that $\#A_d[\phi_d](F) = 3$ if $d$ is a square and $1$ otherwise. The result of this computation is \Cref{table:selmerratio}.
\end{proof}

\section{Selmer groups and integrality}\label{sec:selmer-groups-and-integrality}
We return to the global setting of the introduction, so that $F$ is a number field and $n = 3^m$ for some $m \geq 1$. Recall that $\zeta \in \overline F$ is a primitive $n$-th root of unity.  Let $\phi \colon A \to B$ be a $\zeta$-linear $3$-isogeny over $F$. Recall from \Cref{sec:zeta-mult} that there is a twist $\Aone$ of $A$, and an isogeny $\pi \colon A \to \Aone$, which becomes isomorphic to the endomorphism $1 - \zeta$ over $F(\zeta)$.

For any $F$-algebra $K$, define $\B(K) = K^\times/K^{\times 2n}$. The notation is meant to suggest that $\B$ is the classifying stack $B\mu_{2n}$.  For $d \in \B(F)$, let $\phi_d \colon A_d \to B_d$ be a twist of $\phi$ corresponding to 
\[d \in F^\times/F^{\times 2n} \simeq H^1(F, \mu_{2n}) \to H^1(F, \Aut_{\bar F}(\phi)),\]
as in \S\ref{subsec:twists}.  The Selmer group $\Sel_{\phi_d}(A_d)$ is the subgroup of $H^1(F, A_d[\phi_d])$ consisting of classes whose restriction lies in the image of the Kummer map
\[\partial_{d,v} \colon B_d(F_v)/\phi_d(A_d(F_v)) \hookrightarrow H^1(F_v,A_d[\phi_d])\]
for all places $v$ of $F$. Sometimes we use the notation $\Sel(\phi_d)$ instead of the more clunky $\Sel_{\phi_d}(A_d)$. 
 
The goal of this section is to compute the average size of $\Sel_{\phi_d}(A_d)$ as $d$ varies. The idea is to view Selmer elements as $\SL_2(F)$-orbits of binary cubic forms and then apply geometry-of-numbers counting techniques. To carry this out, we must show that the orbits corresponding to Selmer elements have representatives with bounded denominator. In fact, we will show that this boundedness only holds if $A[\phi]$ is almost everywhere locally a direct summand of $A[\pi]$.

\subsection{Integrality of Selmer elements}\label{subsect:selmerintegral}

We assume for simplicity that $A[\phi]\simeq\Z/3\Z$ as group schemes. This is not really a constraint, since there is always a quadratic twist of $A$ with this property.
This assumption implies that $A_d[\phi_d] \simeq C_d$, where $C_d$ is the order $3$ group scheme cut out by the quadratic field of discriminant $d$, from \Cref{sect: integral}. 

Recall the space $V$ of binary cubic forms from \Cref{sect: integral}. Recall also the set $V(F)_d$ of cubic forms of discriminant $d$, whose $\SL_2(F)$-orbits are in bijection with $H^1(F, C_d) \simeq H^1(F, A_d[\phi_d])$. Similarly, for each place $v$ of $F$, there is a bijection between $\SL_2(F_v)$-orbits on the set $V(F_v)_d$ and $H^1(F_v, A_d[\phi_d])$. Let $V(F_v)^\sol_d$ denote the subset of $V(F_v)_d$ corresponding to classes $\alpha \in H^1(F_v, A_d[\phi_d])$ in the image of $\partial_{d,v}$. Similarly, let $V(F)_d^\mathrm{sel}$ denote the subset of $V(F)_d$ corresponding to classes in $\Sel_{\phi_d}(A_d)$. Define 
\[V(F)^\mathrm{sel} = \bigcup_{0 \neq d \in \O_F} V(F)^\mathrm{sel}_d \hspace{4mm} \mbox{ and } \hspace{4mm} V(F_v)^\sol = \bigcup_{ d \in \O_v(2n)} V(F_v)^\sol_d.\]
where $\O_v$ is the ring of integers in $F_v$, and $\O_v(2n) = \{d \in \O_v : v(d) < 2n\}$. Similarly, define
\[V(F)^\mathrm{sel}_{\mathrm{sq. free}} = \bigcup_{0 \neq d \in \O_F \ \mathrm{sq.free}} V(F)^\mathrm{sel}_d \hspace{4mm} \mbox{ and } \hspace{4mm} V(F_v)^\sol_{\mathrm{sq. free}} = \bigcup_{ d \in \O_v(2)} V(F_v)^\sol_d,\]
where $\O_v(2) = \{d \in \O_v : v(d) < 2\}$ and the union on the left runs over elements $d\in\O_F$ that are squarefree. Of course, these sets depend on the initial choices of $A$ and $\phi$. 

In order to count the $\SL_2(F)$-orbits on $V(F)^\mathrm{sel}$ and $V(F)^\mathrm{sel}_{\mathrm{sq. free}}$ of bounded discriminant, we wish to prove that these orbits have representatives in $V(\O_F)$, or at least representatives with denominators that are uniformly bounded. Since $\SL_2$ has class number $1$, this is ultimately a local question, and it is enough to prove that for almost all $v$, each $\SL_2(F_v)$-orbit in $V(F_v)^\sol$ and $V(F_v)^\sol_{\mathrm{sq. free}}$ has a representative in $V(\O_v)$. For $V(F)^\mathrm{sel}_{\mathrm{sq. free}}$, this integrality holds without any conditions. However, for $V(F)^\mathrm{sel}$, we are forced to assume that $A[\phi]$ is almost everywhere locally a direct summand of $A[\pi]$, as defined in the introduction.

The following integrality result is crucial for the proofs of Theorems \ref{thm:avgSel} and \ref{thm:avgSel2} below.

\begin{theorem}\label{thm:integralityaxiom}
Let $v\nmid 6\infty$ be a place of $F$ at which $A$ has good reduction.
\begin{enumerate}
    \item Every element of $V(F_v)^\sol_{\mathrm{sq. free}}$ is $\SL_2(F_v)$-equivalent to an element of $V(\O_v)$.
    \item If $A[\phi]$ is a direct summand of $A[\pi]$ as a $G_{F_v}$-module, then every element of $V(F_v)^\sol$ is $\SL_2(F_v)$-equivalent to an element of $V(\O_v)$.
\end{enumerate}
\end{theorem}
\begin{proof}
For each $d \in \O_v$, we must show that each class of $H^1(F_v, A_d[\phi_d]) \simeq H^1(F_v, C_d)$ that lies in the image of $\partial_{d,v}$ corresponds to an integral orbit of discriminant $d$. This follows from a comparison of \Cref{thm:integralorbits} with \Cref{thm:soluble orbits} and \Cref{rem:units}. 
\end{proof}

 \subsection{Average size of the Selmer group}\label{sec:avg-size-selmer}
 There is a natural height function on $\B(F)$ defined as follows. Let $M_\infty$ be the set of archimedean places of $F$. If $d \in \B(F)$ with lift $d_0 \in F^\times$, then define the ideal $I = \{a \in F \colon a^{2n} d_0 \in \O_F\}.$
The height of $d$ is then 
\[h(d) = \Nm(I)^{2n} \prod_{v \in M_\infty} |d_0|_v.\]
This is independent of the lift $d_0$, by the product formula. If $F = \Q$, then $h(d) = |d_0|$, where $d_0$ is the unique $2n$-th power free integer representing $d$. For any $X > 0$, there are finitely many $d \in \B(F)$ with $h(d) < X$. 
 
In order to state a robust version of Theorems \ref{thm:selmer} and \ref{thm:squarefreeselmer}, we recall from \cite{elk} the notion of functions on $F$ that are defined by local conditions. Let $F_\infty = \prod_{v \in M_\infty} F_v$. 
We say a function $\psi\colon F \to [0,1]$ is \emph{defined by local congruence conditions} if there exist local functions $\psi_v \colon F_v \to [0,1]$ for every finite place $v$ of $F$, and a function $\psi_\infty \colon F_\infty \to [0,1]$, such that the following two conditions hold:
\begin{enumerate}[label=(\arabic*), leftmargin=*]
\item For all $w \in F$, the product $\psi_\infty(w) \prod_{v \notin M_\infty} \psi_v(w)$ converges to $\psi(w)$. 
\item For each finite place $v$, and for $v = \infty$, the function $\psi_v$ is nonzero on some open set and locally constant outside some closed subset of $F_v$ of measure $0$.
\end{enumerate}
A subset of $F$ is said to be \emph{defined by local congruence conditions} if its characteristic function is defined by local congruence conditions. 

Let $\Sigma_0$ be a fundamental domain for the action of $F^\times$ on $F$ defined by $\alpha\cdot\beta = \alpha^{2n} \beta$. 
We may take $\Sigma_0$ so that it is defined by local congruence conditions. 
For any $X > 0$, let $F_X$ denote the set of $d \in F^\times$ such that $h(d) < X$. Then $\Sigma_0 \cap F_X$ is finite and we think of the abelian varieties $A_d$ as elements of $\Sigma_0$, so that the set of all $A_d$, with $d\in \B(F)$ and $h(d)<X$, is naturally in bijection with the finite set $\Sigma_0 \cap F_X$.

A \emph{family of twists $\{A_d\}$ defined by local congruence conditions} is then a subset $\Sigma \subset \Sigma_0$ defined by local congruence conditions. In that case, the characteristic function $\chi_{\Sigma}$ of $\Sigma$ factors as 
\[\chi_{\Sigma} = \chi_{\Sigma,\infty}\prod_{v \notin M_{\infty}} \chi_{\Sigma_v}.\] For each finite place $v$ of $F$, let $\Sigma_{v}$ be the subset of $F_v$ whose characteristic function is $\chi_{\Sigma_v}$, and let $\Sigma_{\infty}$ be the subset of $F_\infty$ whose characteristic function is $\chi_{\Sigma, \infty}$. 

We say that $\Sigma$ is \emph{large} if $\Sigma_{v}$ contains the set $\O_v(2) := \{d \in \O_v \colon v(d) < 2\}$ for all but finitely many finite places $v$, and if $\Sigma_{\infty}$ is a non-empty union of cosets in $\B(F_\infty)$. By construction, we have $\Sigma_{0,v} = \O_v(2n) \supset \O_v(2)$ for all finite $v$, so $\Sigma_0$ is itself large.

If $f$ is a positive function on $\B(F)$ and $\Sigma \subset \Sigma_0$, we write $\Sigma(X) = \Sigma \cap F_X$ and define
\[\displaystyle\mathrm{avg}_\Sigma \, f(d) = \lim_{X \to \infty}\frac{1}{\#\Sigma(X)}\displaystyle\sum_{d \in \Sigma(X)} f(d).\]
Our formula for $\avg_{\Sigma}\, \Sel_{\phi_d}(A_d)$ is most neatly formulated in terms of the global Selmer ratios $c(\phi_d)$ defined in the introduction, and first defined in \cite{elk}. 

\begin{theorem}\label{thm:avgSel}
Let $\Sigma$ be a large family of twists $A_d$. For each $k\in \Z$, let $T_k \subset \Sigma$ be the subset of $d \in \Sigma$ such that $c(\phi_d) = 3^k$. Assume either that $A[\phi]$ is almost everywhere locally a direct summand of $A[\pi]$ or that $\Sigma_v = \O_v(2)$ for all but finitely many $v$. If $T_k$ is non-empty, then \[\underset{d\in T_k}\avg \#\Sel_{\phi_d}(A_d) = 1 + 3^k.\]
\end{theorem}

When they are non-empty, the sets $T_k$ are countable disjoint unions of large sets. Using the uniformity estimate \cite[Thm.\ 17]{BSW:globalI} and copying the argument from \cite[pp.\ 319-320]{elk}, we deduce \Cref{thm:avgSel} from the following result. To state it, we define for any $d = (d_v)_{v\in M_\infty} \in \B(F_\infty)$,
\[c_\infty(\phi_d) := \prod_{v \in M_\infty} c_v(\phi_{d_v}).\] 
We also let $\overline\Sigma_\infty$ denote the image of $\Sigma_\infty$ in the finite group $\B(F_\infty)$. 

\begin{theorem}\label{thm:avgSel2}
Let $\Sigma$ be a large family of twists $A_d$. Assume either that $A[\phi]$ is almost everywhere locally a direct summand of $A[\pi]$ or that $\Sigma_v = \O_v(2)$ for all but finitely many $v$. Then 
\[\avg_{\Sigma}\, \#\Sel_{\phi_d}(A_d) = 1 + \dfrac{\displaystyle\int_{d \in \overline\Sigma_\infty} c_\infty(\phi_d)\mu_\infty(d)}{\displaystyle\int_{d \in \overline\Sigma_\infty}\mu_\infty(d)}\prod_{v \notin M_\infty} \dfrac{\displaystyle\int_{d \in \Sigma_v} c_v(\phi_d)\mu_v(d)}{\displaystyle\int_{d \in \Sigma_v} \mu_v(d)},\]
where $\mu_v$ is any Haar measure on $\O_v$ and $\mu_\infty$ is the uniform measure on $\B(F_\infty)$. 
\end{theorem}
\begin{proof}
Assume at first that $A[\phi](F) \simeq \Z/3\Z$, as in \Cref{subsect:selmerintegral}. Then the left hand side is equal to the limit as $X \to \infty$ of the average number of $\SL_2(F)$-orbits of discriminant $d$ in $V(F)^\mathrm{sel}$, for $d$ in the finite set $\Sigma \cap F_X$. Given our key result \Cref{thm:integralityaxiom}, \Cref{thm:avgSel2} now follows exactly as in the proof of \cite[Thm.\ 11]{elk}, and we refer the reader there for the details. Note that two proofs were given in that paper, one in the case $F = \Q$ and one for general number fields $F$. The proof over general fields $F$ relies on the geometry-of-numbers machinery developed in the preprint \cite{BSW:globalII}, which has still not appeared. One can alternatively make use of the techniques in \cite{BSW:globalI}, again using \Cref{thm:integralityaxiom} as a key input, to deduce the formula for general $F$. Note that in \cite{BSW:globalI}, the authors count cubic extensions of $F$ with prescribed local conditions (ordered by discriminant), exactly by counting {\it integral} $\SL_2(F)$-orbits of binary cubic forms with prescribed local conditions. 

The proof in the general case where $A[\phi](F) \neq \Z/3\Z$ is exactly the same except we identify elements of $\Sel_{\phi_d}(A_d)$ with orbits of binary cubic forms of discriminant $dk$, where $k \in \O_F$ is chosen so that $F(\sqrt{k}) = F(A[\phi])$, as is done in \cite[\S4]{bkls}. \end{proof}

Taking $\Sigma = \Sigma_0$ in \Cref{thm:avgSel2}, we deduce \Cref{thm:selmer}, and taking $\Sigma = \{d : d\in \O_v(2)\ \forall v\}$, we deduce \Cref{thm:squarefreeselmer}.

\subsection{Explicit Selmer rank bounds}

The following consequence of \Cref{thm:avgSel} will be helpful in giving explicit average rank bounds.

\begin{proposition}\label{prop:tk-rank-bounds}
 Let $\phi\: A\to B$, $\Sigma$, and $T_k$ be as in Theorem $\ref{thm:avgSel}$. Then
 \begin{enumerate}
 \item For each $d \in F^\times$, we have $c(\phi_d) = \dfrac{\#\Sel(\phi_d)}{\#\Sel(\widehat\phi_d)} \cdot \dfrac{\#\widehat B_d[\widehat\phi_d](F)}{\# A_d[\phi_d](F)}.$
  \item If $T_k$ is non-empty, then it has positive density and 
  \[\underset{d\in T_k}{\avg} \dim_{\F_3}\Sel(\phi_d) \oplus \Sel(\widehat\phi_d) \leq |k| + 3^{-|k|}.\]
  \item For a proportion of at least $1 - \frac{1}{2\cdot3^{|k|}}$ of $d \in T_k$, we have $\dim_{\F_3} \Sel(\phi_d) \oplus \Sel(\widehat\phi_d) = |k|$.
 \end{enumerate}
\end{proposition}

\begin{proof}
Since $\phi$ and $\widehat \phi$ determine dual local conditions in their respective Selmer groups \cite[B.1]{Cesnavicius}, the Greenberg--Wiles formula \cite[8.7.9]{NSW:cohomologyofnumberfields} applies and gives $(i)$. 
The argument for $(ii)$ is then exactly the same as \cite[Thm.\ 50]{elk}. 

For $(iii)$, we may assume $k \geq 0$, by switching $\phi$ and $\widehat\phi$ if necessary. By $(i)$, we have $\dim\Sel(\phi_d) + \dim\Sel(\widehat\phi_d) = k$ if and only if $\dim\Sel(\widehat\phi_d) =0$, at least for the $100\%$ of $d$ such that $\#\widehat B_d[\widehat\phi_d](F) = \# A_d[\phi_d](F) = 1$. By \Cref{thm:avgSel}, the average size of $\Sel(\widehat\phi_d)$ for $d\in T_k$ is $1 + 3^{-k}$. Hence, if $s_0$ is the $\liminf$ of the natural density of $d\in T_k$ with $\#\Sel(\widehat\phi_d) = 1$, then
 \[s_0 + 3(1-s_0) \leq 1 + 3^{-k},\]
 and hence $s_0 \geq 1 - \frac{1}{2\cdot3^{|k|}}$.
\end{proof}

\begin{proposition}\label{prop:Selmer-rank-explicit-bound}
    Let $\phi\: A\to B$ and $T_k$ be as in Theorem $\ref{thm:avgSel}$, with $\Sigma$ the set of squarefree twists. Let $S$ be the set of places of $F$ dividing  $3\f_A\infty$, where $\f_A$ is the conductor of $A$. Then $T_k = \emptyset$ if $|k|> \#S$.  As a consequence, we have $\avg_\Sigma\dim_{\F_3}\Sel(\phi_d) \oplus \Sel(\widehat\phi_d) \le  \#S + 3^{-\#S}.$
\end{proposition}

\begin{proof}
    If $v\nmid 3$ is a prime of good reduction, then by \Cref{thm:tamagawa-ratios} and the assumption that $d$ is squarefree, $c_v(\phi_d) = 1$. On the other hand, directly from the definition, we see that $c_v(\phi_d) \geq \frac13$ for any $v$. Hence,
    \[c(\phi_d) = \prod_{v\mid 3\f_A\infty}c_v(\phi_d) \geq \prod_{v\mid 3f_A\infty}\frac13 = 3^{-\#S}.\]
    For almost all $d \in \Sigma$, we have $A_d[\phi_d](F) =0 = \widehat B_d[\widehat\phi_d](F)$, in which case \Cref{prop:tk-rank-bounds}(i) gives
     \[c(\phi_d)c(\widehat\phi_d) = \dfrac{\#\Sel(\widehat\phi_d)}{\#\Sel(\phi_d)}\cdot \dfrac{\#\Sel(\phi_d)}{\#\Sel(\widehat\phi_d)} = 1.\]
    By the above, we also have $c(\widehat\phi_d) \geq 3^{-\#S}$, and hence $c(\phi_d) \leq 3^{\#S}$. Since $3^{-\#S} \leq c(\phi_d) \leq 3^{\#S}$, we have $T_k = \emptyset$ if $|k | >  \#S$. The second claim now follows from \Cref{prop:tk-rank-bounds}$(ii)$.
    \end{proof}

\section{The average rank is bounded in cyclotomic twist families}\label{sec:twists-of-ab-vars}

In this section, we apply Theorems \ref{thm:selmer} and \ref{thm:squarefreeselmer} to prove Theorems \ref{thm:completely-reducible} and \ref{thm:explicit-rank-bound-intro}, i.e.\ to bound the average Mordell--Weil rank of $A_d(F)$.

\begin{proof}[Proof of Theorem $\ref{thm:completely-reducible}(ii)$]
 Since $A[\pi]$ admits a full flag, by \Cref{lem:tensor-cyclotomic}, so does $\Ai[\pi]$ for each $i = 1, \ldots, 2\cdot 3^{m-1}-1$. Hence, for each $i$, there is a sequence of $\zeta$-linear $3$-isogenies
	\[A^{(i)}=B_0^{(i)}\xrightarrow{\phi^{(i)}_1} B_1^{(i)}\to \cdots\to B_{k-1}^{(i)} \xrightarrow{\phi^{(i)}_k} B_k^{(i)} = A^{(i+1)}.\]
	By \Cref{thm:squarefreeselmer}, for each $i,j$, the average rank of $\Sel(\phi^{(i)}_{j,d})$, for squarefree $d\in F^\times/F^{\times 2n}$ is bounded.  
	
	Recall \cite[Lemma 9.1]{bkls}, that if $\psi_1\: A_1\to A_2$ and $\psi_2\:A_2\to A_3$ are isogenies of abelian varieties, then there is an exact sequence
	\begin{equation}\label{eq: selmer-ses}
	 \Sel_{\psi_1}(A_1)\to \Sel_{\psi_2\circ\psi_1}(A_1)\to \Sel_{\psi_2}(A_2).
	\end{equation}
	
	Hence, for each $d$, we have
	\begin{align*}
	\rk A_d(F) &\le\dim_{\F_3}\frac{A_d(F)}{3A_d(F)}\le \dim_{\F_3}\Sel_3(A_d)\\
	&\le \sum_{i =0}^{2\cdot 3^{m-1}-1} \sum_{j=1}^k\dim_{\F_3}\Sel(\phi^{(i)}_{j,d})\\
	&\le\sum_{i = 0}^{2\cdot 3^{m-1}-1} \sum_{j=1}^k\#\Sel(\phi^{(i)}_{j,d}).
	\end{align*}
	Here, the second inequality follows inductively from \Cref{eq: selmer-ses}, and the third inequality follows from the trivial inequality $k\le 3^k$ for all $k\iZ$. Taking the average over $d$, the result follows from \Cref{thm:squarefreeselmer}.
\end{proof}

In the remainder of this section, we prove \Cref{thm:completely-reducible}$(i)$. The proof just given does not apply: by assumption, $B_0^{(i)}[\phi_1^{(i)}]$ is a direct summand of $\Ai[\pi]$, however, there is no reason that $B_{j-1}^{(i)}[\phi_{j}^{(i)}]$ should be a direct summand of $B_{j-1}^{(i)}[\pi]$. Hence, we cannot directly apply \Cref{thm:selmer} to get the result. Instead, we exploit the fact that we can take the isogenies $\phi_{j}^{(i)}$ in any order, together with the following dual version of \Cref{thm:selmer}:

\begin{corollary}\label{lem:last-map-bounded}
	Let $B, C$ be abelian varieties over $F$ with $\zeta$-multiplication and a $\zeta$-linear $3$-isogeny $\phi\:B\to C$ defined over $F$. Suppose that $\widehat C[\widehat\phi]$ is almost everywhere locally a direct summand of $\widehat C[\pi]$. Then the average size of the Selmer group $\Sel(\phi_d)$ is bounded. 
\end{corollary}

\begin{proof}
 Let $\widehat\phi \colon \widehat C \to \widehat B$ be the dual isogeny, which is itself a $\zeta$-linear $3$-isogeny, with respect to the natural $\zeta$-multiplication structure on $\widehat C$ and $\widehat B$. Let $\B = \B_{2n}$ as in \Cref{sec:selmer-groups-and-integrality}. For each $d \in \B(F)$, \Cref{prop:tk-rank-bounds}$(i)$ gives \[\dfrac{\#\Sel(\phi_d)}{\#\Sel(\widehat\phi_d)} = \dfrac{\#B_d[\phi_d](F)}{\#\widehat C_d[\widehat\phi_d](F)}c(\phi_d).\]
The classes $d$ for which $B_d[\phi_d](F) = \Z/3\Z$ or $\widehat C_d[\widehat\phi_d](F)=\Z/3\Z$, form a union of at most two square-classes in $\B(F)$, so we can ignore these values of $d$ when trying to bound the average size of $\Sel(\phi_d)$. In any case, up to a harmless factor of 3, the formula reads
\[\dfrac{\#\Sel(\phi_d)}{\#\Sel(\widehat\phi_d)} = c(\phi_d).\]
By \Cref{thm:selmer} applied to $\widehat\phi$, the average size of $\#\Sel(\widehat\phi_d)$ is bounded. In fact, if we restrict to the set $T_k = T_k(\widehat\phi) \subset \B(F)$ where $c(\widehat\phi_d)= 3^k$, then the average size is equal to $1 + 3^{k}$ by \Cref{thm:avgSel}. Since $c(\phi_d) = 1/c(\widehat\phi_d)$, the average size of $\Sel(\phi_d)$ is equal to $3^{-k} + 1$, and in particular converges, at least on this set $T_k$. 

To see that the average size of $\#\Sel(\phi_d)$ converges (to the expected number) on all of $\B(F)$, we can argue as in \cite[\S6.4]{elk}, using the uniformity estimate \cite[Prop.\ 29]{BST} to give a tail bound for those binary cubic forms with large square part in their discriminant. The uniformity estimate applies to elements of $\#\Sel(\widehat{\phi}_d)$ since, under the hypotheses of the corollary, they are represented by integral binary cubic forms. However, we need to bound the average size of $c(\phi_d)\#\Sel(\widehat{\phi}_d)$ and not $\#\Sel(\widehat{\phi}_d)$, so we also need to control the size of $c(\phi_d)$. By \Cref{table:selmerratio}, we have $c(\phi_d) = O(3^m)$, where $m$ is the number of primes $v$ of $F$ with $v(d)$ even and positive. Moreover, each element of $\Sel(\widehat{\phi}_d)$ corresponds to a binary cubic form whose discriminant has norm divisible by $q_v^2$, for all such $v$. Since 
\[3^m\prod_{v(d) = 2} q_v^{-2} \gg  3q_{v_1}^{-2}\gg q_{v_1}^{-2}, \]
the same argument as in loc.\ cit.\ goes through. We refer there for the details.
\end{proof}

\begin{lemma}\label{injection-of-selmer}
  Let $F$ be any field and suppose there is a commutative diagram of isogenies of abelian varieties over $F$,
	\begin{center}
		\begin{tikzcd}
A \arrow[r, "\phi_1"] \arrow[d,"\phi_2"] & B_1 \arrow[d, "\psi_2"] \\
		B_2 \arrow[r, "\psi_1"]       & {C}
		\end{tikzcd}\end{center}
	such that $\phi_2$ maps $A[\phi_1]$ isomorphically onto $B_2[\psi_1]$. Then $\psi_2$ induces an injection
	\[\frac{B_1(F)}{\phi_1(A(F))}\hookrightarrow\frac{C(F)}{\psi_1(B_2(F))}.\]
	In particular, there is an embedding $\Sel_{\phi_1}(A)\hookrightarrow\Sel_{\psi_1}(B_2)$.
\end{lemma}

\begin{proof}
	Taking cohomology, we obtain the following commutative diagram, with exact rows:
	\begin{center}
		\begin{tikzcd}
		0 \arrow[r] &A[\phi_1](F)\arrow[r] \arrow[d,equal,"\phi_2"] & A(F) \arrow[r, "\phi_1"] \arrow[d,"\phi_2"]        & B_1(F) \arrow[r,"\delta_1"] \arrow[d, "\psi_2"]       & H^{1}(F, A[\phi_1])\arrow[d,equal,"\phi_2"] \\
		0 \arrow[r] & B_2[\psi_1](F)\arrow[r] & B_2(F) \arrow[r, "\psi_1"]       & {C(F)} \arrow[r, "\partial_1"]      & H^1(F, B_2[\psi_1])            & 
		\end{tikzcd}
	\end{center}
	Consider the composite of maps
	\[B_1(F)\xrightarrow{\psi_2}C(F)\to \frac{C(F)}{\psi_1(B_2(F))}.\]
	We show that the kernel is exactly $\phi_1(A(F))$ and, therefore, that there is an injection
	\[\psi_2\colon\frac{B_1(F)}{\phi_1(A(F))}\hookrightarrow\frac{C(F)}{\psi_1(B_2(F))},\]
	from which the result follows.	If $x \in A(F)$, then $\psi_2(\phi_1(x)) = \psi_1(\phi_2(x)) =0.$
	Conversely, if $y \in B_1(F)$ and $\psi_2(y) = \psi_1(z)$ for some $z\in B_2(F)$, then
	\[\phi_2(\delta_1(y))=\partial_1(\psi_2(y))=\partial_1(\psi_1(z)) = 0. \]
	Since $\phi_2$ is an isomorphism on cohomology, it follows that $\delta_1(y) = 0$ and, hence, that $y$ is in the image of $\phi_1$.
\end{proof}

Recall that $A[\pi]$ decomposes as a direct sum of characters, so that $A[\pi](\overline F)=\langle P_1,\ldots,P_k\rangle$, and $\Ga F$ stabilizes each of the subgroups $\langle P_i\rangle$. Thus, $\pi$ factors as a product of $\zeta$-linear $3$-isogenies:
\begin{center}
	\begin{tikzcd}
	A=B_0 \arrow[rrrr, "\pi"] \arrow[rd,"\phi_1"'] &        &          &     &             \Aone=B_k \\
	&B_1 = A/\langle P_1\rangle \arrow[r,"\phi_2"]&\cdots \arrow[r,"\phi_{k -1}"] & {B_{k-1} = A/\langle P_1,\ldots, P_{k-1}\rangle} \arrow[ru,"\phi_k"'] &  
	\end{tikzcd}
\end{center}
Moreover, each of the maps $\phi_i\colon B_{i-1}\to B_{i}$ can be twisted to a map $\phi_{i,d}\colon B_{i-1,d}\to B_{i,d}$.

\begin{corollary}\label{each-selmer-group-bounded}
	For each $i = 1,\ldots, k$, the average size of $\Sel_{\phi_{i,d}}(B_{i-1,d})$, for $d \in \B(F)$, is bounded.
\end{corollary}

\begin{proof}
  The assumption that $A[\pi]$ decomposes as a sum of characters means that we can take the $P_i$'s in any order. Hence, for each $i, d$, there is a commutative diagram
	\begin{center}
	\begin{tikzcd}
	B_{i-1,d} \arrow[r, "\phi_{i,d}"] \arrow[d]        & B_{i,d}  \arrow[d]    \\
	B'_{k-1,d}\arrow[r, "\psi_{i,d}"]       & A^{(1)}_d
	\end{tikzcd}\end{center}
where 
\[B'_{k-1,d}:=\frac{A_{d}}{\langle P_1, \ldots, P_{i-1},P_{i+1}\ldots, P_k\rangle}\]
and $\ker(\psi_{i, d}) = \ker(\phi_{i, d}) = \langle P_i\rangle$.

	 By \Cref{lem:tensor-cyclotomic}, since $A[\pi]$ is completely reducible, so is $\Aone[\pi]$ and hence so is $\widehat A^{(1)}[\pi]$. Thus $\widehat{A}^{(1)}[\widehat\psi_{i,d}]$ is a direct summand of $\widehat A^{(1)}[\pi]$. It follows from \Cref{lem:last-map-bounded} that the average size of $\Sel_{\psi_{i,d}}(B'_{k-1,d})$ is bounded. By \Cref{injection-of-selmer}, we have embeddings $\Sel_{\phi_{i,d}}(B_{i-1,d})\hookrightarrow \Sel_{\psi_{i,d}}(B'_{k-1,d})$ for each $d \in \B(F)$, so the average size of $\Sel_{\phi_{i,d}}(B_{i-1,d})$ is bounded as well.
\end{proof}

\begin{proof}[Proof of Theorem $\ref{thm:completely-reducible}(i)$]
	By \Cref{lem:tensor-cyclotomic}, since $A[\pi]$ is completely reducible, so is $A^{(i)}[\pi]$ for all $i$. Hence, we can factor $\Ai\to A^{(i+1)}$ as
	\[A^{(i)}=B_0^{(i)}\xrightarrow{\phi^{(i)}_1} B_1^{(i)}\to \cdots\to B_{k-1}^{(i)} \xrightarrow{\phi^{(i)}_k} B_k^{(i)} = A^{(i+1)}.\]
	By \Cref{each-selmer-group-bounded}, for each $i,j$, the average rank of $\Sel(\phi^{(i)}_{j,d})$ is bounded. The result now follows exactly as in the proof of \Cref{thm:completely-reducible}$(ii)$.
\end{proof}

\begin{proof}[Proof of Theorem $\ref{thm:explicit-rank-bound-intro}$]
   The hypotheses imply that the Rosati involution associated to the polarization restricts to complex conjugation on the subring $\Z[\zeta]$. Thus, after identifying $A\simeq \widehat A$ and $A_{-3^{n}}\simeq \widehat A_{-3^{n}}$, we can factor multiplication by $-3$ on $A_d$ as the composition
 \[A_d\xrightarrow{\pi_d^{3^{m-1}}} A_d^{(3^{m-1})} = A_{-3^nd}\xrightarrow{\widehat\pi_d^{3^{m-1}}}\widehat A_d,\]
 where the middle equality is \Cref{rem:zeta-twist-vs-d-twist}. As in the proof of \Cref{thm:completely-reducible}$(ii)$, we can factor $\pi_d^{3^{m-1}}$ as a product of $\dim A$ $3$-isogenies $\phi_{j, d}$. By duality, $\widehat\pi_d^{3^{m-1}}$ factors as the product of the dual isogenies $\widehat\phi_{j, d}$.  Thus, for each $d$, we have
 \begin{align*}
	\rk A_d(F) &\leq\dim_{\F_3}\frac{A_d(F)}{3A_d(F)}\leq \dim_{\F_3}\Sel_3(A_d)\\
	&\leq \sum_{j=1}^{\dim A}\dim_{\F_3}\Sel(\phi_{j, d}) \+ \Sel(\widehat\phi_{j, d})
\end{align*}
and hence, by \Cref{prop:Selmer-rank-explicit-bound},
\[\avg_d\rk A_d(F) \leq \dim A\cdot (\#S + 3^{-\#S}).\]
\end{proof}

\section{The average rank in twist families of trigonal Jacobians}\label{sec:trigonal}

Next we use \Cref{thm:completely-reducible} to prove \Cref{cor:trigonal}. In this section, $F$ is a number field and $\zeta \in \overline F$ is a primitive $3^m$-th of unity, for some $m \geq 1$. Let $n = 3^m$, as always.

\subsection{Trigonal Jacobians}\label{subsec:trigonal}
Let $f(x) \in F[x]$ be a monic separable polynomial such that $f(0) \ne 0$, and let $C$ be the smooth projective curve with affine model 
\[C \colon y^3 = x f(x^{3^{m-1}}).\] If $m > 1$, then $C$ has a unique rational point $\infty$ at infinity, and has genus $g = 3^{m-1}\deg(f)$. In fact, we will assume $m > 1$, since  the case $m = 1$ will be subsumed by the results of \Cref{subsec:iteratedPryms}.

Let $J = \Jac(C)$ be the Jacobian of $C$, a $g$-dimensional principally polarized abelian variety over $F$. The automorphism $(x, y)\mapsto (\zeta^3 x, \zeta y)$ induces an automorphism of $J_{\bar F}$ of order $3^m$, which we again call $\zeta$, and which endows $J$ with $\zeta_{n}$-multiplication. As in \Cref{sec:zeta-mult}, the endomorphism $1-\zeta \in \End_{\bar F}(J)$ descends to an isogeny $\pi\:J\to J^{(1)}$ over $F$.

\begin{lemma}\label{lem:group-theory-semisimple}
 Let $G$ be an extension of $(\Z/2\Z)^k$ by a $3$-group $H$. Then every irreducible representation $\rho\:G\to \GL_N(\F_3)$ is one-dimensional. Consequently, any $G$-representation $V$ over $\F_3$ admits a full flag.
\end{lemma}

\begin{proof}
 Since $H\lhd G$ is normal and $\rho$ is semisimple, $\rho|_H$ is also semisimple. Now, any non-trivial representation $V$ over $\F_3$ of a $3$-group contains a non-zero fixed vector \cite{Serre-linear-reps}*{Prop.\ 26}. Thus, by semisimplicity and induction on $\dim V$, we see that $\rho|_H$ is trivial.  Thus,  $\rho$ factors through a representation of $(\Z/2\Z)^k$. For any $g\in (\Z/2\Z)^k$, $\rho(g)\in \GL_N(\F_3)$ has order at most $2$, so its minimal polynomial, either $X\pm1$ or $X^2-1$, has distinct $\F_3$-rational roots. Hence, $\rho(g)$ is diagonalizable, and since $(\Z/2\Z)^k$ is abelian, the operators $\rho(g)$ are simultaneously diagonalizable. In other words, $\rho$ is a direct sum of characters. Since $\rho$ is irreducible, it follows that $\rho$ is one-dimensional.\end{proof}

We first prove \Cref{cor:trigonal} for Jacobians of the curves $C\: y^3 = xf(x^{3^{m-1}})$. In \Cref{thm:iteratedPrym}, we will address the curves $C\: y^{3^m} = f(x)$. 

\begin{proof}[Proof of Corollary $\ref{cor:trigonal}$]
	By assumption $\Gal(f)$ is an extension of $(\Z/2\Z)^k$ by a $3$-group $H$. By \Cref{thm:completely-reducible}, it is enough to show that the Galois representation $J[\pi]$ has a full flag, and splits as a direct sum of characters if $H = 1$.

\begin{lemma}
    $J[\pi^{3^{m-1}}]=J[1-\zeta_3]$ is a maximal isotropic $\F_3$-subspace of $J[3]$ of dimension $g$.
\end{lemma}
 \begin{proof}
     Degree considerations show that $\dim J[1-\zeta_3] = g$, so we need only show that $J[1-\zeta_3]$ is isotropic with respect to the Weil pairing $J[3] \times J[3] \to \mu_3$. Now, the Rosati involution $\dagger$ sends the ideal $(1 - \zeta_3) = (\sqrt{-3})$ to itself (by \Cref{lem:trigonal-rosati} below), and $\langle \alpha P, Q \rangle = \langle P, \alpha^\dagger Q\rangle$ for all $\alpha \in \End(J_{\overline{F}})$ and $P, Q \in J[3]$. If $P,Q \in J[(\sqrt{-3})]$, we may write $Q = \sqrt{-3}(R)$, for some $R \in J[3]$. We then compute
     \[\langle P,Q\rangle = \langle P, \sqrt{-3}(R)\rangle = \langle -\sqrt{-3}(P),R\rangle = 1,\]
     showing that $J[1 - \zeta_3] = J[\sqrt{-3}]$ is isotropic.
 \end{proof}
It follows that $\dim_{\F_3}J[\pi] = \deg(f)$. Explicitly, if $\alpha$ is a root of $f$ and $\beta^{3^{m-1}}=\alpha$, then the divisor
\[(\beta, 0) + (\zeta^3\beta, 0) + (\zeta^{3\cdot 2}\beta,0) +\cdots + (\zeta^{3^m-3}\beta, 0) - 3^{m-1}\infty\]
is fixed by $\zeta$, and $J[\pi]$ is generated by the above divisor classes. Moreover, if $K$ is the splitting field of $f$ over $F$, then each of these divisors defines an element of $J(K)$.
	
The action of $G_F$ on $J[\pi]$ induces a homomorphism $\rho \colon G_F\to \GL_N(\F_3)$, where $N = \deg(f)$, whose kernel is exactly $G_K$. The image is therefore isomorphic to $\Gal(K/F)$, which is an extension of $(\Z/2\Z)^k$ by a $3$-group $H$. Hence, by \Cref{lem:group-theory-semisimple}, $J[\pi]$ admits a full flag. If, moreover, $H = 1$, then $J[\pi]$ is completely reducible, as explained in the proof of \Cref{lem:group-theory-semisimple}.
\end{proof}

\subsection{Iterated triple covers and Pryms}\label{subsec:iteratedPryms}
For our second class of Jacobians, let $f(x) \in F[x]$ be a monic separable polynomial of degree $N > 1$. Let $C$ be the smooth projective curve with affine model $y^{3^m} = f(x).$
The  degree $n = 3^m$ map $C  \to \P^1$, sending $(x,y) \mapsto x$, has Galois group $\mu_n$, at least over $\overline F$.  If $3 \nmid N$, then the fiber above infinity is a unique $F$-rational point and $C$ has genus $g = \frac12(N - 1)(3^m - 1)$. If $3 \mid N$, then the fiber above infinity may have more than one point and they may not be $F$-rational.

Let $J = \Jac(C)$ be the Jacobian. The order $3^m$ automorphism $(x, y)\mapsto (x, \zeta y)$ of $C$ induces an automorphism $\zeta$ on $J$. When $m = 1$, this endows $J$ with $\zeta_3$-multiplication, and we are in a trigonal situation similar to  \Cref{subsec:trigonal}. However, if $m \geq 2$, the automorphism $\zeta \in \Aut_{\bar F}(J)$ does not give rise to a $\zeta_{n}$-multiplication on $J$, as we have defined it in this paper. 
\begin{example}
Consider the curve $C \colon y^9 = x^2 - 1$ of genus $4$. This admits a degree three map to the elliptic curve $E \colon y^3 = x^2 - 1$, and the Jacobian $J = \Jac(C)$ is isogenous to $A \times E$, for some abelian $3$-fold $A \subset J$. The order $9$ automorphism $\zeta$ induces an order $9$ automorphism on $A$ and an order $3$ automorphism on $E$. It thereby endows $A$ with $\zeta_9$-multiplication and $E$ with $\zeta_3$-multiplication, but it does \emph{not} give a $\zeta_9$-multiplication on $J$. Indeed, the minimal polynomial for $\zeta \in \End_{\bar F} J$ is $\Phi_9(x)\Phi_3(x)$ and not $\Phi_9(x) = x^6 + x^3 + 1$. 
\end{example}  

While $J$ does not admit $\zeta_{n}$-multiplication, it is nonetheless isogenous to a product of abelian varieties $P_1 \times P_2 \times \cdots \times P_m$ where each $P_i$ has $\zeta_{3^i}$-multiplication (see \Cref{lem:prym}). In any case, we have $\mu_{2n} \subset \Aut_{\bar F}J$, and we may speak of the twists $J_d$, for each $d \in F^\times/F^{\times2n}$. 

\begin{theorem}\label{thm:iteratedPrym}
	Assume that $\Gal(f)$ is an extension of $(\Z/2\Z)^k$ by a $3$-group $H$. Then the average rank of the twists $J_d$, for squarefree $d\in F^{\times}/F^{\times 2n}$ is bounded. If $H = 1$, then the average rank of the twists $J_d$, for all $d\in F^{\times}/F^{\times 2n}$ is bounded.
\end{theorem}

\begin{proof}
Let $C' \colon y^{3^{m-1}} = f(x)$, and let $J'$ be its Jacobian. Note that when $m = 1$, we have $C' \simeq \P^1$ and $J' = 0$. The map $q \colon C \to C'$ sending $(x,y) \mapsto (x,y^3)$ induces a surjection $q_* \colon J \to J'$, and we let $P$ be the identity component of the kernel, i.e.\ $P$ is the (generalized) Prym variety for the cover $q$. Since $q$ is a ramified triple cover, the map $q^* \colon J' \to J$ is injective. We may identify $q^* = \widehat q_*$, and it follows that the kernel of $q_*$ is already connected, and hence equal to $P$. 

\begin{lemma}\label{lem:prym}
$P$ admits $\zeta_{3^m}$-multiplication. 
\end{lemma}
\begin{proof}
It is enough to show that the minimal polynomial for $\zeta_{3^m}$, as an endomorphism of $P$, is $\Phi_{3^m}(x) = x^{2\cdot 3^{m-1}} + x^{3^{m-1}} + 1$. For this, it is enough to show that the characteristic polynomial of $\zeta_{3^m}$, acting on the homology lattice $H_1(P_\C, \Z)$ is $\Phi_{3^m}(x)^{N-1}$. By \cite[Lem.\ 3.16]{arul}, the characteristic polynomial of $\zeta_{3^m}$ acting on $H_1(C_\C, \Z)\simeq H_1(J_\C,\Z)$ is $(1 + x + x^2 + \cdots + x^{3^m - 1})^{N-1}$. The claim now follows, by induction on $m$.  
\end{proof}
Let $\pi \colon P \to P^{(1)}$ be the isogeny over $F$ descending $1 - \zeta$ over $F(\zeta)$, as usual. Note that $P[\pi] \subset P[3]$ since $P$ has $\zeta_{3^m}$-multiplication, whereas $J[1-\zeta]$ is not in general contained in $J[3]$.

\begin{lemma}\label{lem:prymrep}
The representation $G_F \to \Aut_{\F_3}P[\pi]$ factors through $\Gal(f)$.
\end{lemma}

\begin{proof}
If $\alpha_1, \ldots, \alpha_N$ are the roots of $f(x)$, then the divisor classes $(\alpha_i,0) - (\alpha_j,0)$ generate the group $J[1-\zeta]$ \cite[\S 3]{schaeferprimepower} and so the $G_F$-action on $P[\pi] \subset J[1-\zeta]$ factors through $\Gal(f)$.
\end{proof}

Now we finish the proof of \Cref{thm:iteratedPrym}. By Lemmas \ref{lem:group-theory-semisimple} and \ref{lem:prymrep}, the Galois module $P[\pi]$ has a full flag, and is completely reducible if $H=1$. Thus \Cref{thm:completely-reducible} says that the average rank of $P_d$, for squarefree $d \in F^\times/F^{\times 2 \cdot 3^m}$, is bounded (and without the squarefree condition if $H = 1$). Up to isogeny, we have $J_d \simeq P_d \times J'_d$, where $J'_d$ is the twist of $J'$ (which has $\mu_{2\cdot3^{m-1}}$ twists) by the image of $d$ under $F^\times/F^{\times 2 \cdot 3^m} \to F^\times/F^{\times 2 \cdot 3^{m-1}}$. By induction, the average rank of $J'_d$, for $d \in F^\times/F^{\times 2 \cdot 3^m}$ is bounded. (We view the family $J'_d$ over $\B_{2\cdot 3^m}$, instead of the more natural $\B_{2 \cdot 3^{m-1}}$, but the same proof works for this slightly ``redundant'' family as well.) Thus the average rank of $J_d$ is bounded. 
\end{proof}

\begin{remark}
Combining the two families considered in this section, we obtain similar results for the curves $C_{k,j} \colon y^{3^k} = xf(x^{3^j})$. The Jacobian $\Jac(C_{k,j})$ is isogenous to $\prod_{r = 1}^k P_{r,j}$, where each $P_{r,j} = \mathrm{Prym}(C_{r,j} \to C_{r-1,j})$ is a generalized Prym variety with $\zeta_{3^{r+j}}$-multiplication.
\end{remark}

\section{CM abelian varieties}\label{sec:cm}

Next, we prove \Cref{thm:cmintro}. Let $\zeta = \zeta_{3^m}$ be a primitive $3^m$-th root of unity. We recall our definition of complex multiplication:

\begin{definition}
    An abelian variety $A$ over a number field $F$ \emph{has complex multiplication by $\Z[\zeta]$} if $A$ has dimension $3^{m-1}$ and a $G_F$-equivariant ring embedding $\Z[\zeta] \hookrightarrow \End_{\bar F} A$.
\end{definition}

\begin{proof}[Proof of Theorem $\ref{thm:cmintro}$]
Set $g = 3^{m-1} = \dim A$. The assumption $\Q(\zeta)\subset F$ ensures that $A\simeq A^{(1)}$ by \Cref{lem:zeta-twist}. Hence, we can view the $3$-isogeny $\pi\: A\to A^{(1)}$ as an endomorphism of $A$, and we have $\pi^{2g} = 3u$ for some automorphism $u$ of $A$. By the multiplicativity of the global Selmer ratio \cite[Cor.\ 3.5]{shnidmanRM}, we have 
\begin{equation}\label{eq:selmermult}
 c(\pi_d)^{2g} = c(\pi_d^{2g}) = c([3]u) = c([3]).
 \end{equation}
We claim that $c([3]) = 1$. If $v$ is an infinite prime, then since $F \supset \Q(\zeta)$, we have $F_v \simeq \C$ and $c_v([3]) = \#A[3](\C)^{-1} = 3^{-2g}$. If $v \nmid 3$ is a finite prime then $c_v([3]) = c_v(A_d)/c_v(A_d) = 1$. Finally, $\prod_{v \mid 3} c_v([3]) = 3^{g [F : \Q]}$, by \cite[Prop.\ 3.1]{shnidmanRM}. Let $[F : \Q] = 2N$. Then $F$ has $N$ complex places, so $c([3]) = 3^{-2gN} \cdot 3^{g\cdot 2N} = 1$, as claimed. By (\ref{eq:selmermult}), we also have $c(\pi_d) = 1$ for all $d \in F^\times/F^{\times6g}$. 

It follows from \Cref{thm:avgSel} that the average size of $\Sel_{\pi_d}(A_d)$ is $1 + 1 = 2$. Since $2r \leq 3^r - 1$ for all integers $r$, we have for all $d$:
\[\rk_{\Z[\zeta]} A_d(F) \leq \dim_{\F_3}\Sel_{\pi_d}(A_d) \le \frac12\left( 3^{\rk\Sel_{\pi_d}(A_d)} - 1\right) = \frac12\left( \#\Sel_{\pi_d}(A_d) - 1\right).\]
Thus, the average $\Z[\zeta]$-rank of $A_d(F)$ is at most $\frac12$. The second part of the Theorem follows from \Cref{prop:tk-rank-bounds}$(ii)$.
\end{proof}
Over general number fields $F$, it is no longer true that $c(\pi_d) = 1$ for all $d$, even for abelian varieties with complex multiplication. Moreover, one must consider more than one 3-isogeny to bound the average rank of $A_d(F)$, in general.  However, one can still prove upper bounds which are significantly stronger than \Cref{thm:explicit-rank-bound-intro}.  For example, for CM abelian varieties $A$ of dimension $g = 3^{m-1}$ over $\Q$ with $\zeta_{3^m}$-multiplication, we can show that the average rank of $A_d(F)$ is at most $\frac{13}{9}g$.  Over the totally real field $\Q(\zeta_{3^m} + \overline \zeta_{3^m})$, we can also prove that an explicit positive proportion of twists have $\rk A_d(F)=  0$.  We omit these proofs, since they are somewhat technical and could probably be optimized further. Finally, we remark that the only other result in the literature in this direction that we are aware of is \cite{Diaconu-Tian}, which proves that an infinite but density zero set of twists of the degree $p$ Fermat Jacobian over $\Q(\zeta_p + \overline \zeta_p)$ have rank 0.   

\section{Rational points on hyperbolic varieties}\label{sec:uniform}

\begin{proof}[Proof of Theorem $\ref{thm:uniformcurves}$]
It is not enough to simply invoke \Cref{thm:completely-reducible} and \cite[Thm.\ 4]{kuhne}, since $\Jac(C_d)$, is not the $d$-th sextic twist of $J = \Jac(C)$ in our sense. Indeed, the twists $C_d$ come from the involution $\tau(x,y) = (-x,y)$, which does not induce $-1$ on $\Jac(C)$.  Instead, we consider the Prym variety $P = \ker(\Jac(C) \to \Jac(C/\tau))$, and its dual $\widehat P$. Note that $\zeta$ preserves $P$, and $\tau$ restricts to $-1$ on $P$. Thus, \Cref{thm:completely-reducible} applies to the twist family $P_d$, for $d \in F^\times/F^{\times6}$, and it follows that $\avg_d \rk \widehat P_d(F) = \avg_d \rk P_d(F)$ is bounded.

The inclusion $P_d \hookrightarrow \Jac(C_d)$ induces a surjection $q\: \Jac(C_d) \to \widehat P_d$. Suppose that $C_d(F)$ is non-empty and that $z_0 \in C_d(F)$. Then composing with the Abel--Jacobi map $C_d \hookrightarrow \Jac(C_d)$, using $z_0$ as base point, we obtain a map $j \colon C_d \to \widehat P_d$. We prove that $j$ is injective on points, except possibly at the fixed points of $\tau$ (the points where $x = 0$ and the point(s) at infinity). If $j(w) = j(z)$, then $w - z$ is the pullback of a divisor on $C_d/\tau$.  Thus $\tau(w - z) \equiv w - z$, and so $\tau(w) + z \equiv w + \tau(z)$. But if $D$ is a divisor of degree $2$ on a non-hyperelliptic curve $C$, then by Riemann--Roch, we have $h^0(D) = 2 + 1 - g + h^0(K - D) = 3 - g + g-2 = 1$. Thus, we must have $\tau(w) +z = w + \tau(z)$. If $\tau(w) = \tau(z)$, then $w = z$, as desired. The only other possibility is that $\tau(w) = w$ and $\tau(z) = z$. This proves the claimed injectivity.

Thus, to prove \Cref{thm:uniformcurves}, we may replace $C_d$ with the image of $j$, which is a closed irreducible curve in $\widehat P_d$.  By \cite[Thm.\ 1.1]{GGK}, there is a constant $c$ such that $\#C_d(F) \leq c^{1 + \rk \widehat P_d(F)}$, for all $d$.  But since $\avg_d \rk \widehat P_d(F)$ is bounded, this implies that for all $\epsilon > 0$, there exists $N_\epsilon$ such that $C_d(F) \leq N_\epsilon$ for at least $1 - \epsilon$ of twists $d$.  
\end{proof}

In order to setup the proof (and statement) of \Cref{thm:uniformtheta}, we need to give a precise description of theta divisors for the curves $C$ with affine model $y^3 = f(x)$.  Recall that for any smooth projective curve $C/F$ of genus $g \geq 2$, the symmetric power $C^{(g-1)}$ parameterizes effective divisors $D$ on $C$  of degree $g-1$. Given $\kappa \in \mathrm{Div}(C)$ of degree $g -1$, there is a morphism  $C^{(g-1)}\to \Jac(C)$ sending $D \mapsto D - \kappa$. Its image is the theta divisor, denoted $\Theta = \Theta_\kappa$. The divisor itself depends on $\kappa$, though its class in the N\'eron--Severi group of $\Jac(C)$ does not. If $2\kappa$ is canonical, then $\Theta$ is preserved by the involution $-1$ on $\Jac(C)$, by Riemann--Roch. Such a $\kappa$ exists over $\overline F$, but need not exist over $F$, in general.  If in addition there exists $\mu_{n} \subset \Aut_{\bar F}(C)$ which fixes $\kappa$, then for each $d \in F^\times/F^{\times 2n}$, we may consider the twist $\Theta_d$ of $\Theta$, which is a divisor in $\Jac(C)_d$.  

In our case, let $f \colon C  \to \P^1$ be the degree three map $(x,y) \mapsto x$. The ramification divisor has the form $2D$, and satisfies $K_C = f^*K_{\P^1} + 2D$.  Hence  $\kappa = D - f^*0$ is a half-canonical divisor (over $F$) which is fixed by the $\mu_3$-action. We therefore obtain sextic twists $\Theta_d \subset \Jac(C)_d$, for each $d \in F^\times/F^{\times 6}$,  as in the statement of \Cref{thm:uniformtheta}.  

\begin{proof}[Proof of Theorem $\ref{thm:uniformtheta}$]
This now follows from \Cref{thm:completely-reducible} and \cite[Thm.\ 1.1]{GGK}.
\end{proof}

\section{Abelian surfaces with \texorpdfstring{$\zeta_3$}{zeta\_3}-multiplication}\label{sec:qm-surfaces}

For our final application, we study a family of abelian surfaces with $\zeta_3$-multiplication, arising as Prym varieties. We prove results on the Mordell--Weil groups in sextic twist families of such surfaces, and give applications to explicit uniform bounds on rational points in sextic twist families of bielliptic trigonal curves of genus three (\Cref{thm:genus3example}). For some recent results on rank statistics in larger families of Prym surfaces, see \cite{laga}. 

Let $F$ be a number field and let $f(x) = x^2 + ax + b \in F[x]$ be a quadratic polynomial with non-zero discriminant and $b \neq 0$. Then  $y^3 = f(x^2) = x^4 + ax^2 + b$ is an affine model of a smooth projective plane quartic curve $C$. Note the double cover $\pi \colon (x,y) \mapsto (x^2,y)$ to the elliptic curve $E \colon y^3 = f(x)$. We refer to these genus three curves as {\it bielliptic Picard curves}; see \cite{LagaShnidman}. As in \Cref{sec:uniform}, we consider the Prym variety $P = \mathrm{Prym}_{C/E}$, i.e.\ the kernel of the map $J = \Jac(C) \to E$ induced by Albanese functoriality. The Prym $P$ need not be principally polarized over $\Q$, but it admits a polarization $\lambda \colon P \to \widehat P$ whose kernel is order 4 \cite{MumfordPrym}. The $\zeta_3$-multiplication on $J$ induces $\zeta_3$-multiplication on $P$, and hence we may speak of the sextic twists $P_d$. In fact, $P_d$ is itself the Prym variety of $C_d \colon y^3 = x^4 + adx^2 + bd^2$, which covers the elliptic curve $E_d \colon y^3 = x^2 + adx + bd^2$.

\begin{lemma}
Let $\pi \colon P \to P_{-27}$ denote the descent of $1-\zeta$ to $F$. Then $P[\pi](\overline F) \simeq (\Z/3\Z)^2$ is spanned by $(s,0) - (-s,0)$ and $(t,0) - (-t,0)$, where $\pm s, \pm t$ are the four roots of $f(x^2)$. 
\end{lemma}

In order to apply our result to $P$, we assume that $f(x)$ has linear factors over $F$, so that $P[\pi]$ decomposes as a direct sum of two 1-dimensional Galois modules, corresponding to the quadratic characters $G_F \to \F_3^\times$ cut out by the fields $F(s)$ and $F(t)$. Then \Cref{thm:completely-reducible} says that the average rank of $P_d(F)$ is bounded, and it is interesting to ask whether there is some positive proportion of $d$ with $\rk P_d(\Q) \leq 1$, so that we may apply the Chabauty method. We do not quite prove that such a positive proportion of $d$ exists for general Pryms of this type, but we can prove it in seemingly any given example with the help of  explicit computations. We demonstrate the idea by proving the following result stated in the introduction:

 \begin{customthm}{\ref{thm:genus3example}}
Consider the sextic twist family $C_d \colon y^3 = (x^2 - d)(x^2 - 4d)$ of genus $3$ curves. For at least $\frac13$ of squarefree $d \in \Z$ such that $d \equiv 2$ or $11 \pmod{36}$, we have $\#C_d(\Q) \leq 5$. 
 \end{customthm}

 We will use the following variant of Chabauty's method:

 \begin{theorem}[Stoll]\label{thm:Stollvariant}
 Let $C$ be a smooth projective curve of genus $g \geq 2$ over a number field $F$, and let $H$ be a $G_F$-stable subgroup of $\Aut_{\bar F} C$. Embed $C \hookrightarrow \Jac(C)$ using any positive degree $H$-invariant divisor as basepoint. Suppose there is a quotient $B$ of $\Jac(C)$ such that the composition $\iota \colon C \hookrightarrow \Jac(C) \to B$ is an embedding, and suppose there exists $H \hookrightarrow \Aut_{\bar F} B$, compatible with the $H$-action on $C$, via $\iota$. Then for all but finitely many $H$-twists $C_\xi$ with $\rk B_\xi(F) < \dim B$, we have 
 \[\#C_\xi(F) \leq f_C(\rk B_\xi(F) + g - \dim B) + \#C^\mathrm{triv}_\xi(F) + \#C^{\mathrm{triv},\mathrm{non\text{-}tors}}_\xi(\bar F) \, \backslash \, C_\xi^\mathrm{triv}(F).\]
 Here, $f_C$ is the explicit function defined in \cite[\S3]{stoll:independence} and $C_\xi^\mathrm{triv}$ is the subscheme of points fixed by some non-trivial automorphism in $H$, and $C^{\mathrm{triv},\mathrm{non\text{-}tors}}_\xi$ is the subscheme of trivial points which map to non-torsion points of $B$. 
 \end{theorem}
 \begin{proof}
This is a straightforward generalization of \cite[Thm.\ 5.1]{stoll:independence}, which is the special case where $B = \Jac(C)$. (We have stated an ineffective version of the result,  which is sufficient for our purposes. This is what allows us to use the function $f_C$ as opposed to Stoll's $\tilde{f}_C$.)  The proof is the same, except that instead of the non-degenerate pairing 
\[\Omega(C/F) \times \Jac(C)(F) \otimes \Q \to F\] used in \cite[\S 6]{stoll:independence}, we use the non-degenerate pairing $\Omega(C/F)^B \times B(F) \otimes \Q \to F$, where $\Omega(C/F)^B $ is the image of $\iota^* \colon \Omega(B/F) \to \Omega(C/F)$.
 \end{proof}

We deduce \Cref{thm:genus3example} from \Cref{thm:Stollvariant} and the following theorem, whose proof will occupy the remainder of this section:

\begin{theorem}\label{thm:prymexample}
Let $C \colon y^3 = (x^2 - 1)(x^2 - 4)$, and let $P$ be the corresponding Prym variety. Let $\Sigma$ be the set of squarefree $d \in \Z$ such that $d \equiv 2$ or $11 \pmod{36}$.\footnote{$\Sigma$ is the set of $d$ such that $d, -3d\notin \Q_2^{\times 2}\cup\Q_3^{\times 2}$.} Then the average rank of $P_d$, for $d \in \Sigma$, is at most $\frac73 \approx 2.33$. Moreover, for at least $\frac13$ of $d \in \Sigma$, we have $\rk\, P_d \leq 1$. 
\end{theorem}

\begin{proof}[Proof of Theorem $\ref{thm:genus3example}$]    
The curve $C$ embeds in $B = \widehat P = J/\pi^*E$, the dual of the Prym $P$; see \cite[1.12]{Barth}. We apply Theorem \ref{thm:Stollvariant} to the cyclic group $H$ of order six generated by $(x,y) \mapsto (-x,\zeta_3y)$. We embed $C$ in its Jacobian using the point $\infty = [0 : 1 : 0]$. Consulting \cite[Lem.\ 3.1]{stoll:independence}, we have $f_C(r) \leq 4$ when $C$ is a plane quartic and $r \leq 2$. We have $C^\mathrm{triv}_d(\Q) = \{\infty\}$ for all $d \in \Sigma$, so the second term in \Cref{thm:Stollvariant} is 1. The third term is 0 since all eight of the trivial points on $C$ map to torsion points of $B$. Indeed, the points with $y = 0$ map to 3-torsion points on $\Jac(C_d)$, and if $P = (0,y_0) \in C_d$, then $2P - 2\infty \in \pi^*E_d$, hence $P$ is sent to a $2$-torsion point on $B = J_d/\pi^*E_d$. Altogether we get a bound of $C_\xi(F) \leq 5$ in Stoll's theorem, which combined with Theorem \ref{thm:prymexample} proves \Cref{thm:genus3example}.\end{proof}

\subsection{Bielliptic Picard curves}

Before proving \Cref{thm:prymexample}, we prove some preliminary lemmas.

 \begin{lemma}\label{lem:trigonal-rosati}
    Let $(J,\lambda)$ be the Jacobian of a curve $C$ with a non-trivial automorphism $\zeta$ inducing $\zeta$-multiplication on $J$.  Then the Rosati involution $\alpha \mapsto \lambda^{-1} \widehat\alpha \lambda$ on $\End(J)$ restricts to complex conjugation on $\Z[\zeta] \subset \End(J)$. 
    \end{lemma}
    
    \begin{proof}
    Let $D_0$ be a degree $g-1$ divisor fixed by $\zeta$.  Consider the theta divisor
    \[\Theta = \{D - D_0 : \deg(D) = g-1, \hspace{1mm} D \mbox{ effective}\} \subset J\] 
    and set $\LL = \O_J(\Theta)$. We have $\lambda = \varphi_{\LL} \colon J \to \widehat J$. Since $\Theta$ is fixed by $\zeta$, we have $\zeta^*\LL \simeq \LL$ and hence $\varphi_\LL = \varphi_{\zeta^*\LL} = \widehat\zeta \varphi_{\LL} \zeta$. Rearranging, we see that the Rosati involution sends $\zeta$ to $\zeta^{-1} = \overline\zeta$.
    \end{proof}

\begin{remark}
    The proof shows, more generally, that if $\alpha$ is an automorphism of a curve $C$, and $\alpha^*$ is the induced automorphism of $\Jac(C)$, then the Rosati involution sends $\alpha^*$ to its inverse.
\end{remark}

Now let $C \colon y^3 = x^4 + ax^2 + b$ be a bielliptic Picard curve defined over $\Q$. Let $P$ be the Prym surface for the covering $C \to E$ where $E \colon y^3 = x^2 + ax + b$.  Since $P$ has $\zeta_3$-multiplication, the endomorphism $[-3] \colon P \to P$ factors as $[-3] =\pi_{-27}\circ\pi$, where $\pi \colon P \to P_{-27}$ is the canonical $(3,3)$-isogeny coming from \Cref{lem:descent} (see also \Cref{rem:zeta-twist-vs-d-twist}). Let $\pi_d \colon P_d \to P_{-27d}$ be the sextic twist family of $(3,3)$-isogenies, and let $\widehat\pi_d\:\widehat{P}_{-27d}\to \widehat P_d$ denote the dual isogeny.

\begin{lemma}\label{lem:selmer-dual-twist}
 $\Sel(\pi_{-27d}) \simeq \Sel(\widehat\pi_d)$.
\end{lemma}
\begin{proof}
      Let $C_d\:y^3 = x^4 + adx^2 + bd^2$, let $E_d\: y^3=x^2+adx + bd^2$, and let $J_d =\Jac(C_d)$.\footnote{Note that this is not the same as the $d$-th sextic twist coming from the $\zeta$-multiplication on $J$. The latter is isomorphic to the $d$-th quadratic twist of the Jacobian of $dy^3 = f(x^2)$, and is in general not a Jacobian.} The abelian variety $P_d$ is, by definition, $\ker(J_d\to E_d)$, where the map $J_d\to E_d$ is induced by the double cover $C_d\to E_d$. Let $\lambda_{J}$ denote the principal polarization of $J_d$, and let $\zeta_J$ be the automorphism of $J_d$ induced by the map $(x, y)\mapsto (x, \zeta_3y)$ on $C_d$.  By \Cref{lem:trigonal-rosati} we have a commutative diagram
    \[\begin{tikzcd}
P_d \arrow[r] \arrow[d, "\overline\zeta"] \arrow[rrr, "\lambda_d", bend left] & J_d \arrow[r, "\lambda_J"] \arrow[d, "\overline\zeta"] & \widehat J_d \arrow[r] \arrow[d, "\widehat\zeta"] & \widehat P_d \arrow[d, "\widehat\zeta"] \\
P_d \arrow[r] \arrow[rrr, "\lambda_d"', bend right]                           & J_d \arrow[r, "\lambda_J"']                            & \widehat J_d \arrow[r]                            & \widehat P_d                           
\end{tikzcd}\]
     It follows that $\zeta\ii = \lambda_d\ii\widehat\zeta\lambda_d$ in $\End(P_d)$, and hence 
    \[ (1-\widehat\zeta\ii) \circ\lambda_d= \lambda_d\circ (1-\zeta)\]
    over $\overline{F}$. Over $F$ we must therefore have $\widehat\pi_{-27d}\circ\lambda_d=\lambda_{-27d}\circ\pi_d$, and since $\lambda_d$ is prime-to-$3$, we deduce $\Sel(\pi_d)\simeq \Sel(\widehat\pi_{-27d})$.
\end{proof}

Since $[-3] =\pi_{-27}\circ\pi$, it follows that
\[\rk(P_d) \le \dim_{\F_3}\Sel_3(P_d) \le \dim_{\F_3}(\Sel(\pi_d) \+ \Sel(\pi_{-27d})) = \dim_{\F_3}\Sel(\pi_d) +\dim_{\F_3} \Sel(\widehat{\pi}_d).\]
The following result relates the parity of $\dim_{\F_3}\Sel_3(P_d)$ to the global Selmer ration $c(\pi_d)$.

\begin{proposition}\label{prop:parity}
Let $d \in \Q^\times$ be such that $P_{-27d}[\pi_{-27d}](\Q) = 0$, and write $c(\pi_d) = 3^m$. Then we have the congruence $ \dim_{\F_3}\Sel_3(P_d) \equiv m\pmod2$. 
\end{proposition}
\begin{proof}
 By the Greenberg--Wiles formula, we have $\#\Sel(\pi_d)/\#\Sel(\widehat \pi_d) = c(\pi_d) = 3^m$. Since $[-3] = \pi_d\circ\pi_{-27d}$, we have an exact sequence
  \[0 \to \Sel(\pi_d) \to \Sel_3(P_d) \to \Sel(\pi_{-27d}) \to \frac{\Sha(P_{-27d})[\pi_{-27d}]}{\pi_d(\Sha(P_d)[3])} \to 0.\]
  Exactness on the left is because $P_{-27d}[\pi_{-27d}](\Q) = 0$. By \Cref{lem:selmer-dual-twist}, there is an isomorphism $\Sel(\pi_{-27d})\simeq\Sel(\widehat\pi_d)$, so we see that 
  \begin{equation}\label{eq:paritySha}
m\equiv \dim_{\F_3}\Sel_3(P_d) + \dim_{\F_3}\frac{\Sha(P_{-27d})[\pi_{-27d}]}{\pi_d(\Sha(P_d)[3])}\pmod 2.
  \end{equation}
  Let
  \[\langle\cdot,\cdot\rangle\:\Sha(P_{-27d})\times\Sha(\widehat P_{-27d})\to\Q/\Z\]
  be the Cassels--Tate pairing. Using the polarization $\lambda_{-27d}\:P_{-27d}\to\widehat P_{-27d}$, define
   \[\langle\cdot,\cdot\rangle_\lambda\:\Sha(P_{-27d})\times\Sha(P_{-27d})\to\Q/\Z\]
  by $\langle x,y\rangle_\lambda = \langle x, \lambda_{-27d}(y)\rangle$. As in \cite{shnidmanRM}*{Thms.\ 4.3, 4.4}, if $x\in\Sha(P_{-27d})[\pi_{-27d}]$, then $y$ is in the image of $\pi_d\:\Sha(P_d)\to \Sha(P_{-27d})$ if and only if $\langle x, \lambda_{-27d}(y)\rangle_{\lambda} = 0$ for all $y\in \Sha(P_{-27d})[\pi_{-27d}]$. Thus, the Cassels--Tate pairing $\langle\cdot,\cdot\rangle_\lambda$ restricts to a non-degenerate paring on the finite group $\frac{\Sha(P_{-27d})[\pi_{-27d}]}{\pi(\Sha(P_d)[3])}$. Moreover, since both $P_{d}$ and $P_{-27d}$ are prime-to-3 polarized, this pairing is anti-symmetric, and therefore alternating. The non-degeneracy implies that it has even $\F_3$-rank. Combining with (\ref{eq:paritySha}), we deduce the desired congruence modulo 2. 
\end{proof}

The following general lemma will be used to compute local Selmer ratios below.
\begin{lemma}\label{lem:3adicduality}
Let $\alpha \colon A \to B$ be an isogeny of abelian varieties over a non-archimedean characteristic $0$ local field $F$. Then $c_\ell(\alpha)c_\ell(\widehat\alpha)= \#(\O_F/\deg(\alpha)\O_F).$
\end{lemma}
\begin{proof}
By \cite[B.1]{Cesnavicius}, the groups $B(F)/\alpha A(F)$ and $\widehat{A}(F)/\widehat{\alpha}\widehat{B}(F)$ are orthogonal complements under Tate-Shatz local duality
\[H^1(F, A[\alpha]) \times H^1(F, \widehat{B}[\widehat{\alpha}]) \longrightarrow \Q/\Z.\]
Thus
\[c_\ell(\alpha)c_\ell(\widehat\alpha) = \dfrac{\#B(F)/\alpha A(F)}{\#A(F)[\alpha]} \cdot\dfrac{\#\widehat{A}(F)/\widehat{\alpha}\widehat{B}(F)}{\#\widehat{B}(F)[\widehat{\alpha}]} = \dfrac{\#H^1(F, A[\alpha])}{\#A(F)[\alpha] \cdot\#\widehat{B}(F)[\widehat{\alpha}] } = \#(\O_F/\deg(\alpha)\O_F),\]
where the final equality follows from the Euler-Poincar\'e characteristic formula.
\end{proof}

\subsection{The example} Now specialize to the context of \Cref{thm:prymexample} and the specific curves $C_d\:y^3 = (x^2-d)(x^2-4d)$. 

The isogeny $\pi \colon P \to P_{-27}$ factors as
\[P\xrightarrow{\phi}B\xrightarrow{\psi}P_{-27},\]
where $B = P/\langle (1,0) - (-1,0)\rangle$. Since $(1,0)-(-1,0)\in P[\pi]$, we obtain twists $\phi_d\:P_d\to B_d$ and $\psi_d\:B_d\to P_{-27d}$ by \Cref{lem:zeta-linear}.  

Let us compute the local Selmer ratios for $\phi_d$ and $\psi_d$, for all $d \in \Sigma$ (where $\Sigma$ is as in \Cref{thm:prymexample}). For any $d \in \Q^\times$, we have $c_\infty(\phi_d) = c_\infty(\psi_d)$, since the kernels of $\phi$ and $\psi$ are both $\Z/3\Z$.  Note that $P$ has good reduction at all $p > 3$, since $C$ does. Thus, for $d \in \Sigma$ and for all $p\nmid 6\infty$, by \Cref{thm:tamagawa-ratios}, we have $c_p(\phi_d) = 1 = c_p(\psi_d)$. If $p = 2$, then since $d\equiv 2, 3\pmod 4$, neither $d$ nor $-3d$ is a square in $\Q_2$, so by \Cref{norms-dimension-table}, $c_2(\phi_d) = 1 = c_2(\psi_d)$ as well.  To compute the ratios $c_3(\phi_d)$ and $c_3(\psi_d)$, we use \Cref{lem:3adicduality}.  By multiplicativity, we have
\[c_3(\phi_d)c_3(\psi_d)c_3(\phi_{-27d})c_3(\psi_{-27d}) = c_3([3]) = 9.\] 
Since $d, -27d \notin \Q_3^{\times2}$, all four of these local Selmer ratios are integers, and the same is true for the dual isogenies. Thus, by \Cref{lem:3adicduality}, each ratio must be either $1$ or $3$. Hence, of the four Selmer ratios $c_3(\phi_d), c_3(\psi_d), c_3(\phi_{-27d}), c_3(\psi_{-27d})$, exactly two are 3 and the other two are 1.

\begin{lemma}\label{lem:densities}
If $d \in \Sigma$, then $c_3(\phi_d) \neq c_3(\psi_d)$.
\end{lemma}
\begin{proof}
By \Cref{prop:parity}, the parity of $\dim_{\F_3}\Sel_3(P_d)$ is odd if and only if $c(\pi_d) = c(\phi_d) c(\psi_d)$ is an odd power of $3$. Hence, by our local computations at all other primes $p \neq 3$ given above, the parity of $\dim_{\F_3}\Sel_3(P_d)$ is odd if and only if $c_3(\phi_d) \neq c_3(\psi_d)$. Since $\Jac(C_d)$ is prime-to-3-isogenous to $P_d \times E_d$, we have $\Sel_3(J_d) = \Sel_3(P_d) \oplus \Sel_3(E_d)$. Also, letting $K = \Q(\zeta_3)$ and $X \in \{J_d,P_d,E_d\}$, we have
\[\dim \Sel_3(X) \equiv \dim \Sel_{\pi}(X) + \dim \Sel_{\pi_{-27}}(X_{-27}) = \dim \Sel_\pi(X_K) \pmod 2,\] 
where $\pi$ is the map induced by $1 -\zeta$ on divisors, and $X_K$ is the base change to $K$. It follows that $c(\phi_d) \neq c(\psi_d)$ if and only if $\dim \Sel_\pi(J_{d,K}) - \dim \Sel_\pi(E_{d,K})$ is odd. The latter two $\pi$-Selmer groups can be computed 
in Magma \cite{magma} for any choice of $d$, using the command \texttt{PhiSelmerGroup}. In fact, it is enough to take $d = 2$, since the isomorphism class of $A_d$ over $\Q_3$ depends only on the the image of $d$ in $\Q_3^\times/\Q_3^{\times 6}$, and all elements of $\Sigma$ map to the sixth-power class of $d =2$.\footnote{To check this, use the fact that $\Z_3^{\times6} =1 + 9\Z_3.$} For $d = 2$, we find that $\dim \Sel_\pi(J_{d,K}) - \dim \Sel_\pi(E_{d,K}) = 1$.
\end{proof}

 To recap, for $d \in \Sigma$, we have $c(\phi_d) = c_3(\phi_d)c_\infty(\phi_d)$ and $c(\psi_d) = c_3(\psi_d)c_\infty(\psi_d)$, and we know $c_\infty(\phi_d) = c_\infty(\psi_d)$ (which is equal to 1 or 1/3, depending on the sign of $d$) and $\{c_3(\phi_d), c_3(\psi_d)\} \subset \{1,3\}$.  Combining this with \Cref{lem:densities}, we conclude that  exactly one of $c(\phi_d)$ and $c(\psi_d)$ is 1 and the other is $3^{\pm1}$.  
 
 \begin{proof}[Proof of Theorem $\ref{thm:prymexample}$]
     For all $d$, we have 
\[\rk P_d(\Q) \leq \rk(\Sel(\pi_d) \+\Sel(\widehat\pi_d))\leq \rk(\Sel(\phi_d) \+ \Sel(\widehat \phi_{d})) + \rk(\Sel(\psi_d) \+ \rk \Sel(\widehat\psi_{d})).\] 
For $d \in \Sigma$, we have seen that exactly one of $c(\phi_d)$ and $c(\psi_d)$ is 1 and the other is $3^{\pm 1}$. Thus,  the average rank of $P_d(\Q)$ for $d \in \Sigma$ is at most $(1 + \frac43) = \frac73$, by \Cref{prop:tk-rank-bounds}. This proves the first claim of \Cref{thm:prymexample}.

 Next we show that $\dim_{\F_3} \Sel_3(P_d) = 1$ for at least $\frac13$ of $d \in \Sigma$.  
 Without loss of generality we may assume that $c(\phi_d) = 1$ and $c(\psi_d) = 3^{\pm 1}$.  By \Cref{prop:tk-rank-bounds}$(iii)$, for at least $\frac12$ of $d \in \Sigma$, we have $\Sel(\phi_d) = 0 = \Sel(\widehat\phi_d)$, and for at least $\frac56$ of $d\in \Sigma$, we have $\dim_{\F_3} \Sel(\psi_d) \oplus \Sel(\widehat\psi_{d}) = 1$.  Thus, for at least $\frac56 - \frac12 = \frac13$ of $d \in \Sigma$, we have 
\[\dim_{\F_3} \Sel_3(P_d) \leq \dim \Sel(\phi_d) + \dim \Sel( \widehat\phi_{d}) + \dim \Sel(\psi_d) + \dim \Sel(\widehat \psi_{d}) \leq 1.\]
This implies that $\rk P_d(\Q) \leq 1$ for at least $\frac13$ of $d \in \Sigma$.
\end{proof}

\subsection{More general curves}

It is plausible that for {\it every} Prym $P$ associated to some curve $C_{a,b} \colon y^3 = (x^2 -a)(x^2 - b)$, our method shows that $\rk P_d \leq 1$ for a positive proportion of $d$.  This holds if one can check a certain $3$-adic condition on the numbers $c_3(\phi_d)$ and $c_3(\psi_d)$, exactly as in the proof of \Cref{lem:densities}. This condition is satisfied in all examples we checked, but we do not have a proof in  general. Since the local Selmer ratios $c_3(\phi_{a,b,d})$ and $c_3(\psi_{a,b,d})$ are locally constant as functions on $\Q_3^3 = \{(a,b,d)\}$, we can at least say that this condition holds for a large class of bielliptic Picard curves $C_{a,b}$, with $a$ and $b$ satisfying certain congruence conditions modulo a power of $3$.

In \cite{SW-rank-gain}, we prove that a positive proportion of $P_d$ have rank at most 1, in the case where $a/b$ is a square, using an extra argument which avoids the local $3$-adic computation.  In general, we have the following result, whose proof is an easier version of the argument given above, so we omit it.

\begin{theorem}
Fix $a \in \Q \setminus \{0,\pm 1\}$. For $d\in \Q\t/\Q^{\times 6}$, let $P_{a,d}$ be the Prym surface for the genus three curve $y^3 = (x^2 - d)(x^2 - ad)$. Then $\rk P_{a,d}(\Q) \leq 2$ for a positive proportion of $d$.
\end{theorem}

\section*{Acknowledgments}
We are grateful to Nils Bruin, Jef Laga, Max Lieblich, Dino Lorenzini, Beth Malmskog, Zach Scherr, and Michael Stoll for helpful conversations and correspondences. We thank the referees for their helpful and detailed comments and corrections, and for suggesting to us cleaner proofs of Lemmas \ref{lem:table-h10cd}, \ref{lem:torsion}, and \ref{lem:3adicduality}. The first author was supported by the Israel Science Foundation (grant No.\ 2301/20). The second author was supported by an Emily Erskine Endowment Fund postdoctoral fellowship at the Hebrew University of Jerusalem, by the Israel Science Foundation (grant No. 1963/20), and by the US-Israel Binational Science Foundation (grant No. 2018250).

\bibliography{bibliography}
\bibliographystyle{alpha} 
\end{document}